%% file: arxiv.tex
\pdfoutput=1
\documentclass{amsart}
\usepackage{amssymb,amsfonts, amsthm, amsmath}
\usepackage{txfonts}
\usepackage{enumerate}
\usepackage{mathrsfs}
\usepackage{cite}

\usepackage{graphicx}
\usepackage{cases}

\usepackage{float}

\usepackage{comment}

\input{newcommands.tex}

\title{Number of Systoles of Once-Punctured Torus and Four-Punctured Sphere}

\author{Naoki Hanada}
\address{Department of Mathematical and Computing Sciences, Tokyo Institute of Technology, 2-12-1 O-okayama, Meguro-ku, Tokyo 152-8552, Japan}
\email{naoki.hanada@gmail.com}

\date{\today}

\begin{document}

\begin{abstract}
We compute the number of systoles, the shortest simple closed geodesics and 2-systoles, the second shortest simple closed geodesics on hyperbolic surfaces homeomorphic to once-punctured torus and four-punctured sphere. 
\end{abstract}

\maketitle

\tableofcontents

\input{NumberOfSystoles.tex}

\bibliographystyle{plain}
\bibliography{NumberOfSystoles}

\end{document}

%% file: newcommands.tex
\makeatletter
\newtheorem*{rep@theorem}{\rep@title}
\newcommand{\newreptheorem}[2]{%
\newenvironment{rep#1}[1]{%
 \def\rep@title{#2 \ref{##1}}%
 \begin{rep@theorem}}%
 {\end{rep@theorem}}}
\makeatother

\newreptheorem{theorem}{Theorem}

\newtheorem{thm}{Theorem}[section]
\newtheorem*{thm*}{Theorem}
\newtheorem{cor}[thm]{Corollary}
\newtheorem{prop}[thm]{Proposition}
\newtheorem{lem}[thm]{Lemma}

\theoremstyle{definition}

\newtheorem*{defn*}{Definition}

\theoremstyle{remark}

\newtheorem*{rem*}{Remark}

\makeatletter
\let\c@equation\c@thm
\makeatother
\numberwithin{equation}{section}

\DeclareMathOperator{\Psl}{PSL}
\DeclareMathOperator{\Pgl}{PGL}
\DeclareMathOperator{\Sl}{SL}
\DeclareMathOperator{\Mod}{Mod}
\DeclareMathOperator{\tr}{tr}
\DeclareMathOperator{\Tr}{Tr}
\DeclareMathOperator{\Dfp}{DF}
\DeclareMathOperator{\Hom}{Hom}

\newcommand{\Sf}{\mathcal{S}}

%% file: NumberOfSystoles.tex
\section{Introduction}
The shortest simple closed geodesics on surfaces are called \textit{systole}s. Among them, systoles of hyperbolic surfaces are studied firstly by Schmutz Schaller in the context of the geometry of numbers. The study of systoles of hyperbolic surfaces mainly treats problems of searching surfaces with maximal number or length of systoles for each signature $(g,n)$ ($g$ is genus and $n$ is number of cusps), or asymptotic behavior of number or length of systoles for $(g,n)$ (see \cite{schmutz1,schmutz2,schmutz3,schmutz4,fanoniparlier,busersarnak,balacheffmokoberparlier,parlier2}). 

On the other hand, there are only a few results on exact number of systoles. As for the maximum number of systoles, known exact numbers are only 3, 3, 5 and 12, for signatures $(1,1)$, $(0,4)$, $(1,2)$ and $(2,0)$, which were given by Schmutz Schaller\cite{schmutz1}.

Let $\Sf_{g,n}$ denote the surface with signature $(g,n)$. In this paper, we compute the number of systoles for surfaces in $\mathcal{T}(\Sf_{1,1})$ and $\mathcal{T}(\Sf_{0,4})$, Teichm\"{u}ller spaces of $\Sf_{1,1}$ and $\Sf_{0,4}$. Results for two surfaces are similar. Let $\Sf$ be $\Sf_{1,1}$ or $\Sf_{0,4}$. Firstly we introduce a global coordinate of $\mathcal{T}(\Sf)$ and compute the action of the mapping class group in this coordinate. Secondly we construct a graph $\Gamma$ on $\mathcal{T}(\Sf)$, which is isomorphic to the dual of the Farey tessellation of the hyperbolic plane, namely a properly embedded 3-regular tree invariant under the action of the mapping class group (see Figure \ref{tritreeS11}, \ref{tritreeS04}). 

We obtain the following result:

\begin{reptheorem}{thethm}
For $m \in \mathcal{T}(\Sf)$;
\begin{enumerate}
\item if $m$ is a vertex of $\Gamma$, then $m$ has precisely three different systoles,
\item if $m$ is on an edge of $\Gamma$ but not a vertex, then $m$ has precisely two different systoles,
\item otherwise, $m$ has only one systole.
\end{enumerate}

\end{reptheorem}

In addition, we construct closed subsets $\Delta_0$ and $\Delta_1(\subset \Delta_0)$ of $\mathcal{T}(\mathcal{S})$ (see Figure \ref{deltaS11}, \ref{deltaS04}) and obtain the result for \textit{2-systole}s, the second shortest simple closed geodesics:
\begin{reptheorem}{thethm2}
For $m \in \mathcal{T}(\Sf)$;
\begin{enumerate}
\item if $m\in \Delta_1$, then $m$ has precisely three different 2-systoles,
\item if $m\in \Delta_0 \setminus \Delta_1$, then $m$ has precisely two different 2-systoles,
\item otherwise, $m$ has only one 2-systole.
\end{enumerate}
\end{reptheorem}



\section{Action of Mapping Class Group on Teichm\"{u}ller Space}
\label{sec2}
 \subsection{Global Coordinate of Teichm\"{u}ller Space}
 \label{teichmullerspace}
  \subsubsection{Teichm\"{u}ller Space of Hyperbolic Surface}
  Let $\Sf$ denote $\Sf_{g,n}$ with negative Euler characteristic.
  
   We define $\mathcal{T}(\Sf)$, the Teichm\"{u}ller space of $\Sf$ as follows:
   \[
   \mathcal{T}(\Sf) \coloneqq \{(\Sigma, f)\}/{\sim}
   \]
   where $\Sigma$ is a surface with a complete hyperbolic structure homeomorphic to $\mathcal{S}$, $f$ is a homeomorphism $\Sf \rightarrow \Sigma$, and the equivalence relation $\sim$ is defined as $(\Sigma,f) \sim (\Sigma' ,f')$ if there is an isometry $\iota \colon \Sigma \rightarrow \Sigma'$ so that $\iota \circ f$ and $f'$ are homotopic.
   
   
   There is an alternate definition of $\mathcal{T}(\Sf)$:
   \[
   \Dfp (\pi _1 (\Sf), \Psl (2,\mathbb{R}))/{\Pgl (2, \mathbb{R})}.
   \]
   $\Dfp (\pi _1 (\Sf), \Psl (2,\mathbb{R}))$ is the set of representations $\pi _1(\Sf) \rightarrow \Psl (2, \mathbb{R})$ that are faithful, discrete, and send elements that are homotopic to punctures to parabolic elements. $\Pgl(2,\mathbb{R})$ acts on it  by conjugation. Restrictions we put on representations correspond to the condition for $\mathbb{H}/{\rho (\pi _1 (\Sf))}$ ($\rho \in \Dfp (\pi _1 (\Sf), \Psl (2,\mathbb{R}))$) to have complete hyperbolic structures on $\Sf$ (see \cite{bonahon, farbmarglit}). 
   
   Let $[(\Sigma , f)] \in \mathcal{T}(\Sf)$. Then we can identify $\pi_1(\Sigma)$ with a deck transformation group of $\tilde{\Sigma}$, the universal cover of $\Sigma$. On the other hand, $f$ identifies $\pi_1(\Sf)$ with $\pi_1(\Sigma)$. Therefore we can define a representation $f_\ast \colon \pi_1(\Sf) \rightarrow \Psl(2,\mathbb{R})$, which identifies $\pi_1(\Sf)$ with the deck transformation group of $\tilde{\Sigma}$. Then $f_\ast $ is an element of $\Dfp (\pi _1 (\Sf), \Psl (2,\mathbb{R}))$ and the map $\mathcal{T}(\Sf) \rightarrow \Dfp (\pi _1 (\Sf), \Psl (2,\mathbb{R}))/{\Pgl (2, \mathbb{R})}$ defined by $[(\Sigma, f)] \mapsto [f_{\ast}]$ is a bijection (see \cite{farbmarglit}).
   
   We identify $\mathcal{T}(\Sf)$ with $\Dfp (\pi _1 (\Sf), \Psl (2,\mathbb{R}))/{\Pgl(2,\mathbb{R})}$ by this bijection. Namely for $[(\Sigma,f)] \in \mathcal{T}(\Sf)$ and corresponding $[\rho] \in \Dfp (\pi _1 (\Sf), \Psl (2,\mathbb{R}))/{\Pgl(2,\mathbb{R})}$, we write $[(\Sigma, f)]=[\rho]$.
   
   Let $\gamma _1, \gamma_2, \cdots ,\gamma_n$ be generators of $\pi _1 (\Sf)$ and
   \[
   R(\Sf) \coloneqq \{(\rho (\gamma_1), \rho (\gamma _2), \cdots ,\rho (\gamma_n)) \mid \rho \in \Dfp (\pi _1 (\Sf), \Psl (2,\mathbb{R})) \} \subset (\Psl (2,\mathbb{R}))^n.
   \]
   There is a natural bijective correspondence between the points of $R(\Sf)$ and the points of $\Dfp (\pi _1 (\Sf), \Psl (2,\mathbb{R}))$. We shall identify them. We give $(\Psl (2,\mathbb{R}))^n$ the topology as a Lie group, $R(\Sf)$ the subspace topology, $R(\Sf)/\Pgl(2,\mathbb{R})$ the quotient topology and $\mathcal{T}(\Sf)$ the induced topology.
   
   We let $\pi \colon R(\Sf) \rightarrow \mathcal{T}(\Sf)$ denote the projection $\rho \mapsto [\rho]$. Then $\pi$ is a fiber bundle with fiber $\Pgl(2,\mathbb{R})$.
   

   \subsubsection{Teichm\"{u}ller Space as Character Variety}
   
   Any discrete subgroup of $\Psl(2,\mathbb{C})$ which has no 2-torsion lifts to $\Sl(2,\mathbb{C})$ (see \cite{culler}). Therefore, for every $\rho \in R(S)$, we can take a lift $\tilde{\rho}$, a representation $\pi_1(\Sf) \rightarrow \Sl(2,\mathbb{R})$ such that $[\tilde{\rho}(\gamma)]=\rho(\gamma)$ where $\gamma \in \pi_1(\Sf)$. Then the set $\hat{R}(\Sf)\coloneqq \{\tilde{\rho} \mid \tilde{\rho} \text{ is a lift of } \rho \in R(\Sf)\}$ is a covering space of $R(\Sf)$ where the quotient map $\hat{R}(\Sf) \rightarrow \hat{R}(\Sf)/\Hom(\pi_1(\Sf),\{\pm I\})=R(\Sf)$ is a covering map (the action of $\delta \in \Hom(\pi_1(\Sf),\{\pm I\})$ is given by $\delta \tilde{\rho}(\gamma)=\delta(\gamma)\tilde{\rho}(\gamma)$). Moreover, $\hat{R}(\Sf)$  is the union of disjoint connected components, each of which is mapped homomorphically onto $R(\Sf)$ by the covering map. Therefore taking a point $\rho \in R(\Sf)$,  a lift $\tilde{\rho}$ of $\rho$ (the number of choices of $\tilde{\rho}$ is $|\Hom(\pi_1(\Sf),\{\pm I\})|$), and the connected component containing $\tilde{\rho}$, we can take a section of the covering. 
   
    On the other hand, for $\Sf_{g,n}$ with $p_1,p_2,\cdots p_n \in \pi_1(\Sf_{g,n})$ homotopic to punctures, we can take $\tilde{\rho}$ so that $\tr (\tilde{\rho}(p_i))=-2$ for $i\in\{1,2, \cdots n\}$ (see \cite{okumura},\cite{saito}). Then the number of choices of $\tilde{\rho}$ is reduced to $|\Hom(\pi_1(\Sf_{g,0}),\{\pm I\})|$. 
    
   For $\rho$, we take such a lift $\tilde{\rho}$ and connected component of $\hat{R}(\Sf)$ containing $\tilde{\rho}$. We let $\tilde{R}(\Sf)$ denote the connected component. Then $R(\Sf)/\Pgl(2,\mathbb{R})$ corresponds to $\tilde{R}(\Sf)/\mathrm{GL}(2,\mathbb{R})$. From now on, for $\rho \in R(\Sf)$, we let $\tilde{\rho}$ denote the lift of $\rho$ such that $\tilde{\rho} \in \tilde{R}(\Sf)$.

   The \textit{character} of $\tilde{\rho} \in \tilde{R}(\Sf)$ is the function $\chi_{\tilde{\rho}} \colon \pi_1(\Sf) \rightarrow \mathbb{R}$ defined by $\gamma \mapsto \tr(\tilde{\rho}(\gamma))$. We consider the set of characters $\{ \chi_{\tilde{\rho}} \mid \tilde{\rho} \in \tilde{R}(\Sf)\}$. In fact there exist finite elements $ \gamma _1', \gamma _2' ,\cdots , \gamma _m' $ of $\pi_1(\Sf)$ such that there is a bijective correspondence between the points of 
   \[
   X(\Sf) \coloneqq \{ (\chi _{\tilde{\rho}}(\gamma_1'), \chi _{\tilde{\rho}}(\gamma_2'), \cdots ,\chi_{\tilde{\rho}}(\gamma_m')) \mid \tilde{\rho} \in \tilde{R}(\Sf) \} \subset \mathbb{R}^m
   \]
   and the points of $\{ \chi_{\tilde{\rho}}\mid \tilde{\rho} \in \tilde{R}(\Sf)\}$ (see \cite{cullershalen}). 
   
   We shall identify the points of $\{ \chi_{\tilde{\rho}} \mid \tilde{\rho} \in \tilde{R}(\Sf)\}$ with the corresponding points of $X(S)$. Since trace is a conjugacy invariant, we can define the map $\mathcal{T}(\Sf) \rightarrow X(\Sf)$ by $[\rho] \mapsto \chi_{\tilde{\rho}}$.

   \begin{lem}
   \label{charactervariety}
   The map $\mathcal{T}(\Sf) \rightarrow X(\Sf)$ defined by $[\rho] \mapsto \chi_{\tilde{\rho}}$ is a bijection.
   \end{lem}
   \begin{proof}
   Surjectivity is obvious. Therefore we should show that if $\chi_{\tilde{\rho}}=\chi_{\tilde{\rho} '}$, there exists $h \in \mathrm{GL}(2,\mathbb{R})$ such that $\tilde{\rho}' = h \tilde{\rho} h^{-1}$. If $\tilde{\rho}$ is irreducible, the above claim holds (see \cite{cullershalen}). Therefore we should show that $\rho(\pi_1(\Sf))$ does not have common fixed points on $\mathbb{R}\cup \{\infty\}$. 
   
   Suppose that $\rho(\pi_1(\Sf))$ has a common fixed point on $\mathbb{R}\cup \{\infty\}$. Then $\mathbb{H}/\rho(\pi_1(\Sf))$ is a hyperbolic cylinder or an incomplete space. This is a contradiction.
   \end{proof}
   
  \subsubsection{Gluing Construction}
  \label{paramet}
  
  We identify Poincar\'{e} extensions of orientation preserving isometries on $\mathbb{H}$ with elements of $\Psl(2,\mathbb{R})$. We write $R_{1,1} = R(\Sf_{1,1})$ and $R_{0,4} = R(\Sf_{0,4})$. For $a,b\in \mathbb{R}\cup\{\infty\}$, we let $[a,b]$ denote the hyperbolic line joining $a$ and $b$. 
  
  Firstly we consider the case of $\Sf_{1,1}$. Let $u$ and $v$ be generators of $\pi_1(\Sf_{1,1})$ such that $u^{-1}v^{-1}uv$ is homotopic to the puncture. 
  
  Suppose that $\rho \in R_{1,1}$. Then $\rho(u^{-1}v^{-1}uv)$ is parabolic. We let $c$ be the fixed point of $\rho(u^{-1}v^{-1}uv)$. Then $[c,\rho(v)(c)]$, $[\rho(v)(c), \rho(uv)(c)]$, $[\rho(uv)(c),\rho(v^{-1}uv)(c)]$ and $[\rho(v^{-1}uv)(c),c]$ form an ideal quadrilateral. Obviously this ideal quadrilateral is a fundamental domain of $\rho(\pi_1(\Sf_{1,1}))$. By gluing $[c,\rho(v)(c)]$ to $[\rho(uv)(c),\rho(v^{-1}uv)(c)]$ by $\rho(u)$ and  $[\rho(v^{-1}uv)(c),c]$ to $[\rho(v)(c), \rho(uv)(c)]$ by $\rho(v)$, we get a hyperbolic surface isometric to $\mathbb{H}/{\rho(\pi_1(\Sf_{1,1}))}$.
  
  Conversely, for an ideal quadrilateral, if we can take $\varphi$ and $\psi$, Poincar\'{e} extensions of isometries on $\mathbb{H}$ so that they map each opposite side of the quadrilateral and $\varphi ^{-1} \psi ^{-1} \varphi \psi$ is parabolic, then there exists $\rho \in R_{1,1}$ such that $\rho(u)=\varphi$ and $\rho(v)=\psi$ and we can identify $(\varphi,\psi)$ with $\rho$. We call such $\varphi$ and $\psi$ \textit{gluing map}s. We can identify $R_{1,1}$ with
  \[
  \{(\varphi,\psi) \mid \varphi,\psi \text{ are gluing maps of an ideal quadrilateral} \}.
  \]


  Consider an ideal quadrilateral $Q(\lambda)$ on $\mathbb{H}$ with ideal vertices $\lambda,0,1,\infty$ where $\lambda <0$. Let $e_1 \coloneqq [1, \infty]$, $e_2 \coloneqq [\lambda, 0]$, $e_3 \coloneqq [\lambda, \infty]$ and $e_4 \coloneqq [0, 1]$. If we can take gluing maps $\varphi$ and $\psi$ for $Q(\lambda)$ so that 
  \begin{equation}
  \begin{cases}
  \varphi(\infty)=\lambda,\ \varphi(1)=0,\  \psi(\infty)=1,\  \text{ and } \psi(\lambda)=0  \\
  \varphi ^{-1} \psi ^{-1} \varphi \psi \text{ is parabolic},
  \label{gluingcondition}
  \end{cases}
  \end{equation}
 then we get an element of $R_{1,1}$ (as we will see in \ref{parS11}, we can take $\varphi$ and $\psi$ for each $\lambda<0$ and they are not uniquely determined).  
    
    We let
    \[
    R^{\ast}_{1,1} \coloneqq \{(\varphi, \psi) \in R_{1,1} \mid \varphi \text{ and } \psi \text{ satisfies condition (\ref{gluingcondition}) where }\lambda <0\}.
    \]
    
    \begin{figure}[h]
  \centering
  \includegraphics[width=6cm]{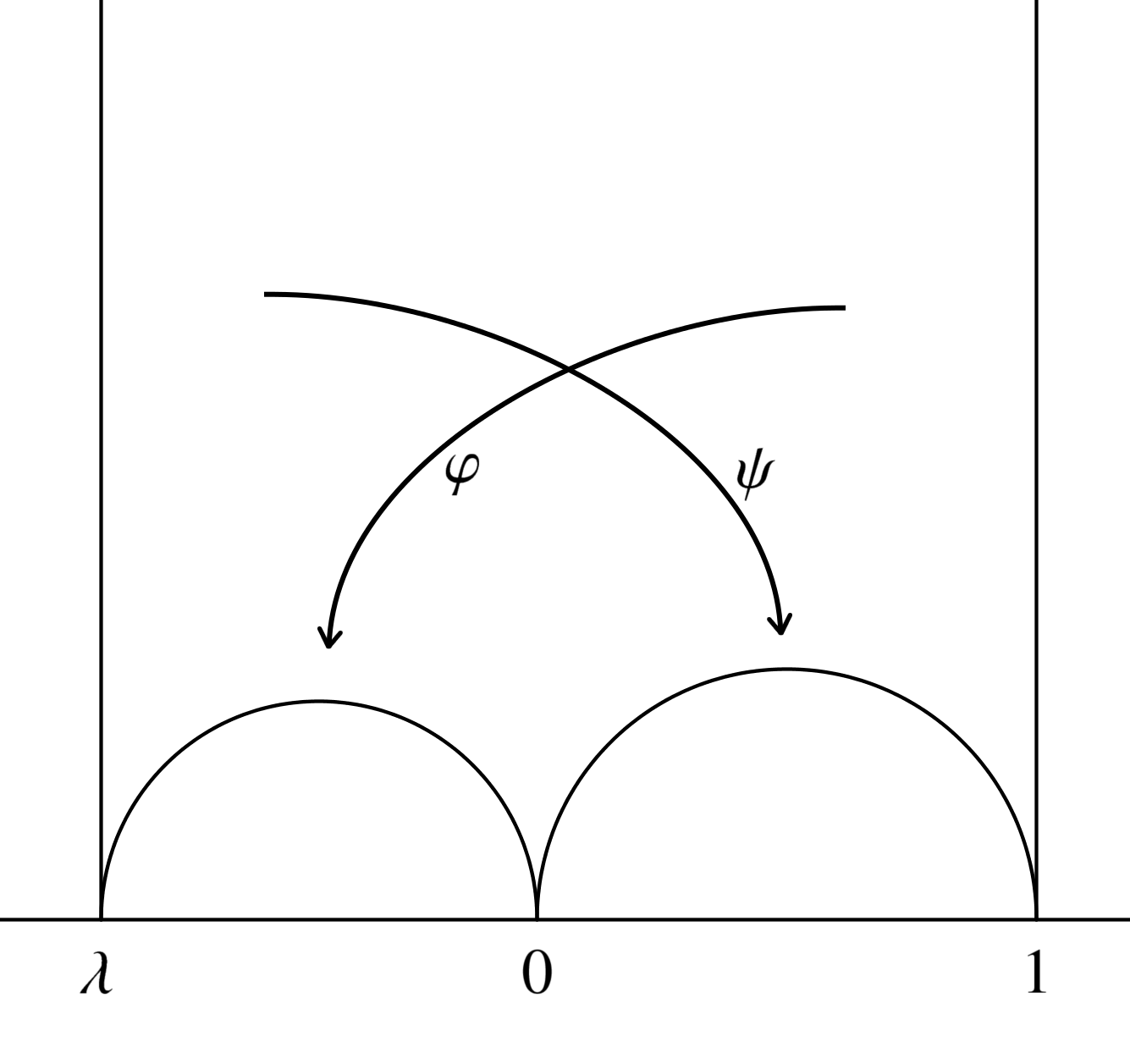}
  \caption{$Q(\lambda)$ and gluing maps.}
  \label{s11}
  \end{figure}
    
   \begin{lem}
   For any $(\varphi',\psi') \in R_{1,1}$, there exist $h\in \Pgl(2,\mathbb{R})$ such that $(h\varphi' h^{-1},h\psi' h^{-1}) \in R^{\ast}_{1,1}$.
   \end{lem}
   \begin{proof}
   Suppose that $(\varphi',\psi') \in R_{1,1}$. Then ${\varphi '} ^{-1} {\psi '} ^{-1} \varphi ' \psi' $ is parabolic.  We let $c$ be the fixed point of ${\varphi'} ^{-1} {\psi'} ^{-1} \varphi' \psi' $ on $\mathbb{R}\cup \{ \infty \}$. Let $e_1' \coloneqq [\psi'(c),c]$, $e_2' \coloneqq \varphi'(e_1')$, $e_3' \coloneqq [{\psi'} ^{-1} \varphi' \psi'(c),c]$ and $e_4' \coloneqq \psi '(e_3')$. Then $e_1',e_2',e_3'$ and $e_4'$ form an ideal quadrilateral $Q'$ and we can construct a complete hyperbolic once-puncutered torus by gluing $e_1'$ to $e_2'$ by $\varphi '$ and $e_3'$ to $e_4'$ by $\psi '$. Then there exists $h\in \Pgl(2,\mathbb{R})$ and $\lambda$ such that $h(Q')=Q(\lambda)$ and $h(e_i')=e_i$ for $i=1,2,3,4$. Then $h\varphi' h^{-1}$ and $h\psi' h^{-1}$ are gluing maps of $h(Q')$. Therefore $(h\varphi' h^{-1},h\psi' h^{-1}) \in R^{\ast}_{1,1}$.
   \end{proof}

   \begin{lem}
   For any $(\varphi,\psi) \in R^{\ast}_{1,1}$, there does not exist $h \in \Pgl(2,\mathbb{R})\setminus \{I\}$ such that $(h\varphi h^{-1},h\psi h^{-1}) \in R^{\ast}_{1,1}$.
   \end{lem}
   \begin{proof}
   Suppose that $h \in \Pgl(2,\mathbb{R})\setminus \{I \}$ and $(\varphi,\psi) \in R^{\ast}_{1,1}$ satisfies $(h\varphi h^{-1},h\psi h^{-1}) \in R^{\ast}_{1,1}$. Then $h$ fixes $0,1$ and $\infty$. Since $h$ fixes three points, $h$ is the identity map. This is a contradiction.
   \end{proof}
    
    Recall that $\pi \colon R(\Sf)\rightarrow \mathcal{T}(\Sf)$ is a fiber bundle with fiber $\Pgl(2,\mathbb{R})$.
 \begin{cor}
 For any $[\rho] \in \mathcal{T}(\Sf_{1,1})$, there uniquely exists $\rho' \in R^{\ast}_{1,1}$ such that $\pi(\rho')=[\rho]$.
 \end{cor}
 Therefore we can define a section of $\pi$ by $[\rho]\mapsto \rho'$ where $\rho' \in R^{\ast}_{1,1}$ and $\pi(\rho')=[\rho]$. $R^{\ast}_{1,1}$ is the image of the section.

  Next we consider the case of $\Sf_{0,4}$. Let $u$, $v$, $w$ be generators of $\pi_1(\Sf_{0,4})$ such that $u$, $v$, $w$ and $(wvu)^{-1}$ are homotopic to distinct punctures and have the same orientation (see Figure \ref{generators}).
  
   Suppose that $\rho \in R_{0,4}$. Then $\rho(u)$, $\rho(v)$, $\rho(w)$ and $\rho((wvu)^{-1})$ are parabolic and have the same orientation. We let $c_1$, $c_2$, $c_3$ and $c_4$ be fixed points of them. Then $[c_1,\rho(u)(c_4)]$, $[\rho(u)(c_4),c_2]$, $[c_2,\rho(vu)(c_4)]$, $[\rho(vu)(c_4),c_3]$, $[c_3,c_4]$ and $[c_4,c_1]$ form an ideal hexagon and this is a fundamental domain of $\rho(\pi_1(\Sf_{0,4}))$. Gluing $[c_4,c_1]$ to $[c_1,\rho(u)(c_4)]$ by $\rho(u)$, $[\rho(u)(c_4),c_2]$ to $[c_2,\rho(vu)(c_4)]$ by $\rho(v)$ and $[\rho(vu)(c_4),c_3]$ to $[c_3,c_4]$ by $\rho(\omega)$, we get a hyperbolic surface isometric to $\mathbb{H}/{\rho(\pi_1(\Sf_{0,4}))}$.
  
  \begin{figure}[h]
  \centering
  \includegraphics[width=5cm]{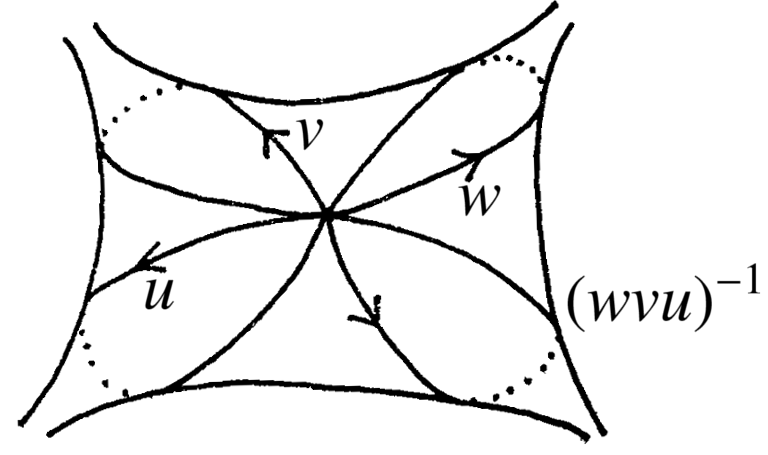}
  \caption{Generators of $\pi_1(\Sf_{0,4})$.}
  \label{generators}
\end{figure}
  
  Conversely, for an ideal hexagon with a vertex $v_i$ and an edge $e_i=[v_i,v_{i+1}]$ where $i \in \mathbb{Z}/{6\mathbb{Z}}$, if we can take $\varphi$, $\psi$ and $\omega$, Poincar\'{e} extensions of isometries on $\mathbb{H}$ so that $\varphi$ maps $e_1$ to $e_2$, $\psi$ maps $e_3$ to $e_4$, $\omega$ maps $e_5$ to $e_6$ and $\varphi$, $\psi$, $\omega$, $\omega \psi \varphi$ are parabolic with fixed points $v_2,v_4,v_6,v_1$, then there exists $\rho \in R_{0,4}$ such that $\rho(u)=\varphi$, $\rho(v)=\psi$ and $\rho(w)=\omega$. Then we can identify $(\varphi,\psi,\omega)$ with $\rho$. We call such $\varphi$, $\psi$ and $\omega$ \textit{gluing map}s. We can identify $R_{0,4}$ with
  \[
  \{(\varphi,\psi,\omega) \mid \varphi,\psi \text{ and } \omega \text{ are gluing maps of an ideal hexagon} \}.
  \]
  
  Consider an ideal hexagon $H(\kappa,\lambda,\mu)$ with edges $[\kappa,\lambda]$, $[\lambda,0]$, $[0,1]$, $[1,\mu]$, $[\mu,\infty]$ and $[\infty,\kappa]$ where $\kappa<\lambda<0<1<\mu$. If we can take gluing maps $\varphi$, $\psi$ and $\omega$ for $H(\kappa,\lambda,\mu)$ so that 
  \begin{equation}
  \begin{cases}
  \label{gluingcondition2}
  \varphi(1)=1,\ \varphi(\mu)=0,\ \psi(\lambda)=\lambda,\ \psi(0)=\kappa,\  \omega(\infty)=\infty \text{ and } \omega(\kappa)=\mu\\
   \varphi,\psi,\omega\text{ and }\omega \psi \varphi \text{ are parabolic},
  \end{cases}
  \end{equation}
 then we get an element of $R_{0,4}$ (as we will see in \ref{parS04}, such $\varphi$, $\psi$ and $\omega$ exist where $\kappa=\lambda \mu$ and they are uniquely determined for each $(\lambda,\mu)$).
  
  We let 
  \[
    R^{\ast}_{0,4} \coloneqq \{(\varphi, \psi,\omega) \in R_{0,4} \mid \varphi , \psi ,\omega \text{ satisfies condition (\ref{gluingcondition2}) where }\kappa<\lambda<0<1<\mu
    \}.
    \]
      
   \begin{figure}[h]
    \centering
  \includegraphics[width=9cm]{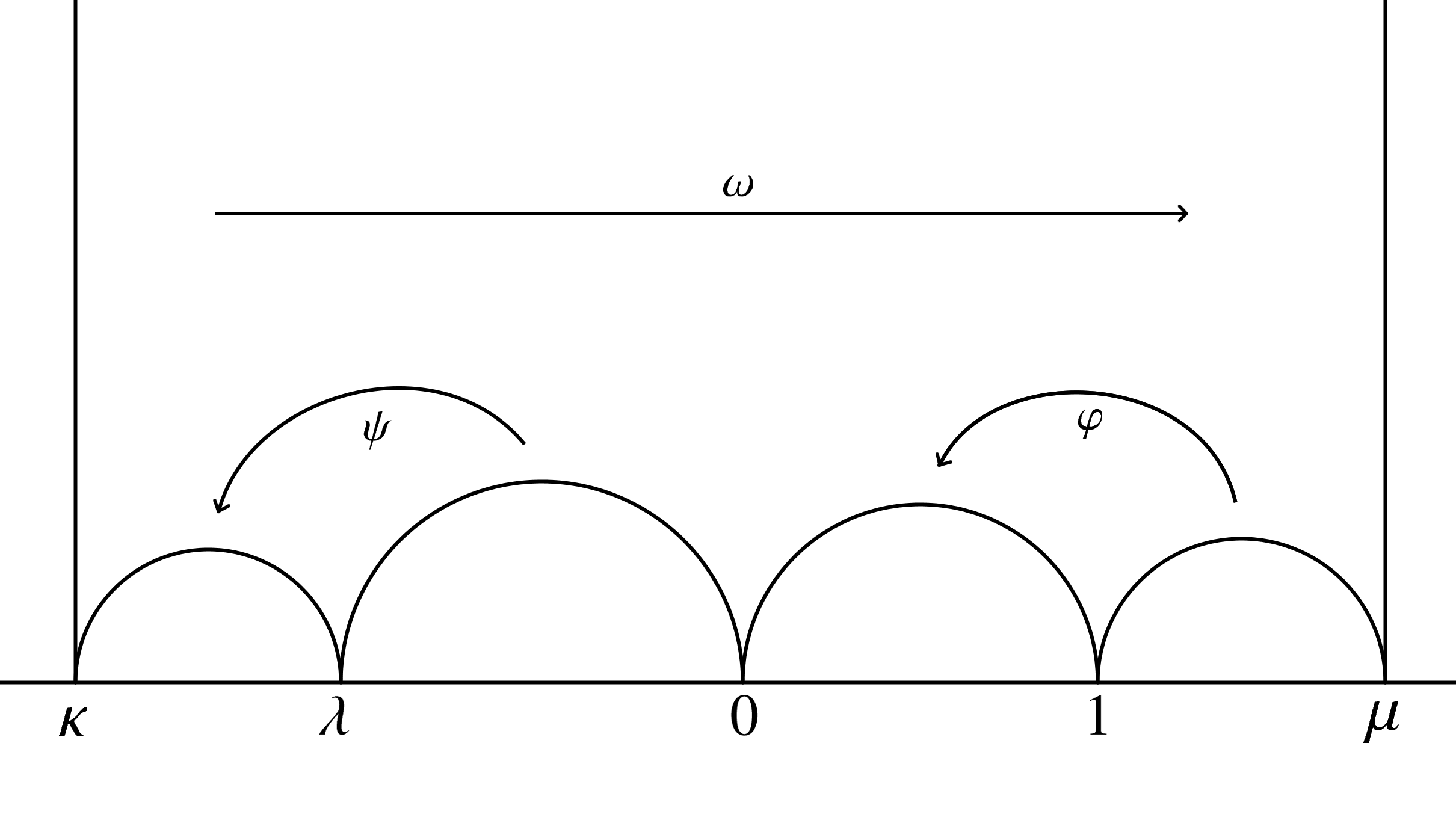}
  \caption{$H(\kappa,\lambda,\mu)$ and gluing maps.}
  \label{s04}
\end{figure}
    By a similar argument to that of $R^{\ast}_{1,1}$, we see that $R^{\ast}_{0,4}$ is the image of a section of $\pi$. We will give coordinates to $R^{\ast}_{1,1}$ and $R^{\ast}_{0,4}$ in \ref{parS11} and \ref{parS04}.

   \subsubsection{$\Sf_{1,1}$}
   \label{parS11}
   
   We write $\tilde{R}_{1,1}=\tilde{R}(\Sf_{1,1})$ and $X_{1,1} = X(\Sf_{1,1})$. We search gluing maps $\varphi$, $\psi$ satisfying condition (\ref{gluingcondition}) in \ref{paramet} for an ideal quadrilateral $Q(\lambda)$. We let
   \[
   \tilde{\varphi} \coloneqq \begin{pmatrix} a & b \\ c & d \end{pmatrix},\  \tilde{\psi} \coloneqq \begin{pmatrix} a' & b' \\ c' & d' \end{pmatrix} \  \in \Sl(2,\mathbb{R})
   \]
    so that $\tilde{\varphi}, \tilde{\psi}$ represent $\varphi, \psi$, and $\tr \tilde{\varphi} ^{-1} \tilde{\psi} ^{-1} \tilde{\varphi} \tilde{\psi} = -2$. Then the condition which $\tilde{\varphi}$ and $\tilde{\psi}$ must satisfy is as follows:
   \[
   \begin{cases}
    \det \tilde{\varphi} = \det \tilde{\psi} = 1 \\
    \lambda < 0 \\
   \varphi (\infty) = \lambda, \varphi (1) = 0, \psi (\infty) = 1, \psi (\lambda) = 0  \\
    \tr \tilde{\varphi} ^{-1} \tilde{\psi} ^{-1} \tilde{\varphi} \tilde{\psi} = -2.
   \end{cases}
   \]
  We get
    \begin{align*}
    &b=-a,\  c=-a'^2/a,\  d=(1+a')/a,\\
    &b'=a^2/a',\  c'=a',\  d'=(1+a^2)/a',\  \lambda = -(a/a')^2.
    \end{align*}
    Replace variables as $\alpha = a, \beta = a'$. Then,
   \[
   \tilde{\varphi} = 
    \begin{pmatrix}
    \alpha & -\alpha \\
    -\frac{\beta ^2}{\alpha} & \frac{1+\beta ^2}{\alpha}
    \end{pmatrix}
    ,\ 
    \tilde{\psi} = 
    \begin{pmatrix}
    \beta & \frac{\alpha ^2}{\beta} \\
    \beta & \frac{1+\alpha ^2}{\beta}
    \end{pmatrix}
   \]
   where $\alpha,\beta \in \mathbb{R}\setminus\{0\}$.

   Sometimes we write $\tilde{\varphi} (\alpha ,\beta) = \tilde{\varphi}$ and $\tilde{\psi} (\alpha ,\beta) = \tilde{\psi}$ regarding them as parametrized by $(\alpha ,\beta)$ and for $M \in \langle \tilde{\varphi}(\alpha,\beta),\tilde{\psi}(\alpha,\beta) \rangle$, we write $M(\alpha,\beta)$. Then we can write
   \[
   R^\ast_{1,1}=\{([\tilde{\varphi} (\alpha,\beta)], [\tilde{\psi} (\alpha,\beta)]) \mid \alpha,\beta \in \mathbb{R}\setminus\{0\}\}.
   \]
   
    On the other hand, $\{(\tilde{\varphi} (\alpha,\beta), \tilde{\psi} (\alpha,\beta)) \mid \alpha,\beta \in \mathbb{R}\setminus\{0\}\} \subset \hat{R}(\Sf_{1,1})$ has four connected components (remark that $|\Hom(\pi_1(\Sf_{1,0}),\{\pm I\})|=4$), namely $\{(\tilde{\varphi} (\alpha,\beta), \tilde{\psi} (\alpha,\beta)) \mid \alpha,\beta >0\}$, $\{(\tilde{\varphi} (\alpha,\beta), \tilde{\psi} (\alpha,\beta)) \mid \alpha>0,\beta<0\}$, $\{(\tilde{\varphi} (\alpha,\beta), \tilde{\psi} (\alpha,\beta)) \mid \alpha,\beta <0\}$ and $\{(\tilde{\varphi} (\alpha,\beta), \tilde{\psi} (\alpha,\beta)) \mid \alpha<0,\beta >0\}$. Each of them is the image of a section of the covering $\{(\tilde{\varphi} (\alpha,\beta), \tilde{\psi} (\alpha,\beta)) \mid \alpha,\beta \in \mathbb{R}\setminus\{0\}\} \rightarrow R^\ast_{1,1}$. We can let $\tilde{R}_{1,1}$ be the connected component of $\hat{R}(\Sf_{1,1})$ containing $\{(\tilde{\varphi} (\alpha,\beta), \tilde{\psi} (\alpha,\beta)) \mid \alpha,\beta >0\}$.
   
   Then we can define the continuous injection $p_{1,1} \colon \mathbb{R}_+^2 \rightarrow R_{1,1}$ by 
   \[
   (\alpha,\beta) \longmapsto ([\tilde{\varphi} (\alpha,\beta)], [\tilde{\psi} (\alpha,\beta)]).
   \] 
   Then $p_{1,1}(\mathbb{R}_+^2)=R^\ast_{1,1}$. Therefore $\pi \circ p_{1,1} \colon \mathbb{R}_+^2 \rightarrow \mathcal{T}(S_{1,1})$ is a homeomorphism. We define 
   \[
   \Phi_{1,1} \coloneqq (\pi \circ p_{1,1})^{-1} \colon \mathcal{T}(\Sf_{1,1}) \rightarrow \mathbb{R}_+^2
   \]
    as a global coordinate of $\mathcal{T}(\Sf_{1,1})$.
  
  We define $\tilde{p}_{1,1} \colon \mathbb{R}_+^2 \rightarrow \tilde{R}_{1,1}$ by 
  \[
  (\alpha,\beta) \longmapsto (\tilde{\varphi} (\alpha,\beta), \tilde{\psi} (\alpha,\beta))
  \]
   and $X_{1,1}$ by
    \[
    \{ (\tr \tilde{\rho} (u),\tr \tilde{\rho}(v), \tr \tilde{\rho} (uv)) \mid \tilde{\rho} \in \tilde{R}_{1,1}\}.
    \]
     We can verify the well-definedness of $X_{1,1}$ by using the trace identity: $\tr AB +\tr AB^{-1}=\tr A \tr B$. 

  \begin{cor}
   The map $\mathbb{R}_+^2 \rightarrow X_{1,1}$ defined by
  \begin{align*}
  (\alpha, \beta) & \mapsto \chi_{\tilde{p}_{1,1}(\alpha,\beta)}\\
  & = (\tr \tilde{\varphi} (\alpha, \beta),\tr \tilde{\psi} (\alpha, \beta), \tr \tilde{\varphi} \tilde{\psi}(\alpha, \beta)) \\
  &=\left(\frac{\alpha^2+\beta^2+1}{\alpha},\frac{\alpha^2+\beta^2+1}{\beta},\frac{\alpha^2+\beta^2+1}{\alpha \beta}\right)
  \end{align*} 
  is a bijection.
  \end{cor}

  \subsubsection{$\Sf_{0,4}$}
  \label{parS04}

  We write  $\tilde{R}_{0,4}=\tilde{R}(\Sf_{0,4})$ and $X_{0,4} = X(\Sf_{0,4})$. We search gluing maps $\varphi$, $\psi$, $\omega$ satisfying condition (\ref{gluingcondition2}) in \ref{paramet} for an ideal hexagon $H(\kappa,\lambda,\mu)$. We let
  \[
  \tilde{\varphi} := \begin{pmatrix} a & b \\ c & d \end{pmatrix},\  \tilde{\psi} \coloneqq \begin{pmatrix} a' & b' \\ c' & d' \end{pmatrix},\  \tilde{\omega} \coloneqq \begin{pmatrix} -1 & -\mu+\kappa \\ 0 & -1 \end{pmatrix} \ \in \Sl(2,\mathbb{R}),
  \] 
   so that $\tilde{\varphi}$, $\tilde{\psi}$, $\tilde{\omega}$ represent $\varphi$, $\psi$, $\omega$, and $\tr \tilde{\varphi}=\tr \tilde{\psi}= \tr \tilde{\omega} = \tr  \tilde{\omega}\tilde{\psi} \tilde{\varphi}  = -2$. Then the condition which $\tilde{\varphi}$, $\tilde{\psi}$ and $\tilde{\omega}$ must satisfy is as follows:
  \[
   \begin{cases}
    \det \tilde{\varphi} = \det \tilde{\psi} = 1 \\
     \varphi(1)=1,\ \varphi(\mu)=0,\ \psi(\lambda)=\lambda,\ \psi(0)=\kappa\\
    \kappa < \lambda < 0<1<\mu \\
    \tr \tilde{\varphi}=\tr \tilde{\psi}=\tr  \tilde{\omega} \tilde{\psi} \tilde{\varphi}  = -2.
   \end{cases}
   \]
 We get 
  \begin{align*}
  &b=-a-1,\  c=a+1,\ d=-a-2,\\
   &a'= -a-2,\  b'=(a+1) \lambda,\  c'=- (a+1)/\lambda,\  d'=a,\\
   &\kappa=(a+1)\lambda/a,\ \mu=(a+1)/a,\\
    &a>0,\ \lambda <0.
  \end{align*}
   Replace variables as $\alpha = a, \beta = - \lambda$. Then
   \[
   \tilde{\varphi} = 
  \begin{pmatrix}\alpha & -\alpha-1\\
\alpha+1 & -\alpha-2\end{pmatrix}
    ,\ 
    \tilde{\psi} = 
   \begin{pmatrix}-\alpha-2 & - \left( \alpha+1\right) \beta\\
\frac{\alpha+1}{\beta} & \alpha\end{pmatrix}
    ,\ 
    \tilde{\omega} = 
    \begin{pmatrix}
    -1 & - \frac{(\alpha +1)(\beta +1)}{\alpha } \\
    0 & -1
    \end{pmatrix}
   \]
   where $\alpha,\beta>0$.
   
   Sometimes we write $\tilde{\varphi} (\alpha ,\beta) = \tilde{\varphi}$, $\tilde{\psi} (\alpha ,\beta) = \tilde{\psi}$ and $\tilde{\omega} (\alpha ,\beta) = \tilde{\omega}$ regarding them as parametrized by $(\alpha ,\beta)$ and for $M \in \langle \tilde{\varphi}(\alpha,\beta) ,\tilde{\psi}(\alpha,\beta), \tilde{\omega}(\alpha,\beta) \rangle$, we write $M(\alpha,\beta)$. Then we can write
   \[
   R^\ast_{0,4}=\{([\tilde{\varphi} (\alpha , \beta)], [\tilde{\psi} (\alpha ,\beta)], [\tilde{\omega} (\alpha ,\beta)]) \mid \alpha, \beta>0\}.
   \]
   
   On the other hand, $\{(\tilde{\varphi} (\alpha,\beta), \tilde{\psi} (\alpha,\beta), \tilde{\omega}(\alpha,\beta)) \mid \alpha,\beta>0\} \subset \hat{R}(\Sf_{0,4})$ is connected and the covering map $\{(\tilde{\varphi} (\alpha,\beta), \tilde{\psi} (\alpha,\beta), \tilde{\omega}(\alpha,\beta)) \mid \alpha,\beta>0\} \rightarrow R^\ast_{0,4}$ is a homeomorphism (remark that $|\Hom(\pi_1(\Sf_{0,0}),\{\pm I\})|=1$). We can let $\tilde{R}_{0,4}$ be the connected component of $\hat{R}(\Sf_{0,4})$ containing $\{(\tilde{\varphi} (\alpha,\beta), \tilde{\psi} (\alpha,\beta), \tilde{\omega}(\alpha,\beta)) \mid \alpha,\beta>0\}$.
   
   Then we can define the map $p_{0,4} \colon \mathbb{R}_+^2 \rightarrow R_{0,4}$ by 
   \[
   (\alpha , \beta) \mapsto ([\tilde{\varphi} (\alpha , \beta)], [\tilde{\psi} (\alpha ,\beta)], [\tilde{\omega} (\alpha ,\beta)]).
   \]
     Then $p_{0,4}$ is a continuous injection and $p_{0,4}(\mathbb{R}_+^2)=R^\ast_{0,4}$. Therefore $\pi \circ p_{0,4} \colon \mathbb{R}_+^2 \rightarrow \mathcal{T}(\Sf_{0,4})$ is a homeomorphism. We define 
     \[
     \Phi_{0,4} \coloneqq (\pi \circ p_{0,4})^{-1} \colon \mathcal{T}(\Sf_{0,4}) \rightarrow \mathbb{R}_+^2
     \]
      as a global coordinate of $\mathcal{T}(\Sf_{0,4})$.
  
  We define $\tilde{p}_{0,4} \colon \mathbb{R}_+^2 \rightarrow \tilde{R}_{0,4}$ by 
  \[
  (\alpha,\beta) \longmapsto (\tilde{\varphi} (\alpha,\beta), \tilde{\psi} (\alpha,\beta), \tilde{\omega}(\alpha,\beta))
  \]
   and $X_{0,4}$ by
   \[
   \{ (\tr \tilde{\rho} (uv), \tr \tilde{\rho} (vw),\tr \tilde{\rho}(wu)) \mid \tilde{\rho} \in \tilde{R}_{0,4}\}.
   \] 
   Well-definedness of $X_{0,4}$ follows from a similar argument to that of $X_{1,1}$.
  \begin{cor}
  The map $\mathbb{R}_+^2 \rightarrow X_{0,4}$ defined by
  \begin{align*}
  (\alpha,\beta) & \mapsto \chi_{\tilde{p}_{0,4}(\alpha,\beta)}\\
   &= (\tr \tilde{\varphi} \tilde{\psi} (\alpha, \beta) ,\tr \tilde{\psi} \tilde{\omega}(\alpha, \beta)  ,\tr \tilde{\omega} \tilde{\varphi}(\alpha, \beta) )\\
  &=\left( 2-\frac{{{\left( \alpha+1\right) }^{2}}{{\left( \beta+1\right) }^{2}}}{\beta},2-\frac{{{\left( \alpha+1\right) }^{2}}\,\left( \beta+1\right) }{\alpha\beta},2-\frac{{{\left( \alpha+1\right) }^{2}}\,\left( \beta+1\right) }{\alpha} \right)
  \end{align*}
  is a bijection.
  \end{cor}

 \subsection{Action of Mapping Class Group}
 \label{actionofmcg}
  \subsubsection{Action on Teichm\"{u}ller Space}
  We define $\Mod(\Sf)$, the mapping class group of $\Sf$ as the group of homotopy classes of orientation-preserving self-homeomorphism on $\Sf$. Remark that if $\Sf$ has punctures, an element of $\Mod(\Sf)$ may send a puncture to a different puncture. 
  
  $\Mod(\Sf)$ acts on $\mathcal{T}(\Sf)$ as follows:
  \[
  g[(\Sigma,f)] = [(\Sigma, f \circ g^{-1})]
  \]
  where $g \in \Mod (\Sf)$ and $[(\Sigma,f)] \in \mathcal{T}(\Sf)$. 
  
   $\Mod(\Sf)$ acts on  $R(\Sf)$ as follows:
  \[
  g\rho = \rho \circ {g_{\ast}}^{-1}
  \]
  where $\rho \in R(\Sf)$ and $g_{\ast}$ is an automorphism of $\pi _1 (\Sf)$ induced by $g \in \Mod(\Sf)$. 
  
  The action on $\tilde{R}(\Sf)$ is given similarly. 
  The action on $R(\Sf)/{\Pgl(2,\mathbb{R})}$ is given by
  \[
  g[\rho] = [g \rho]
  \]
  and on $X(\Sf)$ by
  \[
  g \chi_{\tilde{\rho}} = \chi_{g\tilde{\rho}}.
  \]
   

   Let $\tilde{p}$ denote $\tilde{p}_{1,1}$ or $\tilde{p}_{0,4}$ and $\Phi$ denote $\Phi_{1,1}$ or $\Phi_{0,4}$. We will compute the corresponding action on $\mathbb{R}_+^2$:
   \[
   g(\alpha,\beta)=\Phi(g\Phi^{-1}(\alpha,\beta)).
   \]
  Suppose that $(\alpha',\beta')=g(\alpha,\beta)$. Then 
   \[
   g\chi_{\tilde{p}(\alpha,\beta)}=\chi_{\tilde{p}(\alpha' ,\beta')}
   \]
   holds for each $g \in \Mod(\Sf)$. We should express $(\alpha',\beta')$ in terms of $\alpha$ and $\beta$.
   
  \subsubsection{$\Sf_{1,1}$}
  \label{actofS11}
  \begin{prop}
  $
  \Mod(\Sf_{1,1}) \approx \Sl (2,\mathbb{Z}).
   $
  \end{prop}
  \begin{proof}
  See \cite{farbmarglit}.
  \end{proof}
  $\{\pm I\} \triangleleft \Sl (2,\mathbb{Z})$ corresponds to the subgroup of $\Mod(\Sf_{1,1})$ generated by the hyperelliptic involution on $\Sf_{1,1}$. This is the kernel of the action (see \cite{farbmarglit}). Therefore, we should consider the action of $\Mod(\Sf_{1,1})/{\{ \pm I \}} \approx \Psl(2,\mathbb{Z})$.
  
    For generators $u$, $v$ of $\pi _1 (\Sf_{1,1})$ as in \ref{paramet}, we can take Dehn twists around $u$ and $v$ as generators of $\Mod(\Sf_{1,1})/{\{ \pm I \}}$. We let $\sigma$ and $\tau$ denote Dehn twists around $v$ and $u$. We can suppose that 
    \[
    \sigma _{\ast}^{-1} (u)=uv,\  \sigma_{\ast}^{-1} (v)=v,\  \tau_{\ast}^{-1}(u)=u,\  \tau_{\ast}^{-1}(v)=uv
    \]
     where $\sigma_\ast$ and $\tau_\ast$ are automorphisms of $\pi_1(\Sf_{1,1})$ induced by $\sigma$ and $\tau$. Therefore for any $(\alpha,\beta) \in \mathbb{R}_+^2$, 
     \begin{align*}
      \sigma  \tilde{p}_{1,1}(\alpha,\beta) &= (\tilde{\varphi} \tilde{\psi}(\alpha,\beta) , \tilde{\psi}(\alpha,\beta)),\\
      \tau  \tilde{p}_{1,1}(\alpha,\beta) &= (\tilde{\varphi}(\alpha,\beta) ,\tilde{\varphi} \tilde{\psi}(\alpha,\beta))
      \end{align*}
      holds. Therefore, 
  \begin{align*}
  \sigma \chi_{\tilde{p}_{1,1}(\alpha,\beta)} &=\chi_{ \sigma  \tilde{p}_{1,1}(\alpha,\beta)}=(\tr \tilde{\varphi} \tilde{\psi}(\alpha,\beta) , \tr \tilde{\psi} (\alpha,\beta), \tr \tilde{\varphi} \tilde{\psi} ^2(\alpha,\beta)),\\
   \tau \chi_{\tilde{p}_{1,1}(\alpha,\beta)}&= \chi_{ \tau  \tilde{p}_{1,1}(\alpha,\beta)}=(\tr \tilde{\varphi}(\alpha,\beta), \tr \tilde{\varphi} \tilde{\psi}(\alpha,\beta), \tr \tilde{\varphi}  ^2\tilde{\psi}(\alpha,\beta)).
   \end{align*} 
   Suppose that $(\alpha',\beta')=\sigma(\alpha,\beta)$ and $(\alpha'',\beta'')=\tau(\alpha,\beta)$. Then
  \begin{align*}
  \sigma \chi_{\tilde{p}_{1,1}(\alpha ,\beta)}&=\chi_{\tilde{p}_{1,1}(\alpha ', \beta ')} ,\\
    \tau \chi_{\tilde{p}_{1,1}(\alpha ,\beta)}&=\chi_{\tilde{p}_{1,1}(\alpha '', \beta '')}
  \end{align*}
  holds. We should express $(\alpha',\beta')$ and $(\alpha'', \beta'')$ in terms of $\alpha$ and $\beta$.
  
  We get
  \[
   (\alpha ', \beta ')=\left(\frac{\alpha \beta}{\alpha ^2+1},\frac{\beta}{\alpha ^2+1}\right), \ (\alpha '',\beta '')=\left(\frac{\alpha}{\beta ^2+1},\frac{\alpha \beta}{\beta ^2+1}\right).
  \]
  Letting $S\coloneqq \tau^{-1} \sigma \tau^{-1}$ and $T\coloneqq \tau$, we get
  \[
   S  (\alpha, \beta) = \left(\frac{\beta}{\alpha ^{2} + \beta ^{2}}, \frac{\alpha}{\alpha ^{2} + \beta ^{2}}  \right),\ 
     T (\alpha, \beta) = \left(\frac{\alpha}{1 + \beta ^{2}}, \frac{\alpha \beta}{1 + \beta ^{2}}  \right).
  \]
  Then $S^{2} = (ST)^{3} = \mathrm{id}$.
  
    Apply a coordinate transformation $t\colon (\alpha , \beta) \mapsto (-\log \alpha \beta, \sqrt{3} \log \alpha / \beta) \eqqcolon (A,B)$. Then $t(\mathbb{R}_+^2)=\mathbb{R}^2$. $\Mod(\Sf_{1,1})$ acts on $\mathbb{R}^2$ as follows:
  \begin{align*}
  S (A, B) &=\left(-A -\frac{2}{ \sqrt{3}} B +2 \log \left(  1+ e^{2B /\sqrt{3}} \right), -B \right),\\
  T (A, B)&=\left( -\frac{1}{2} A - \frac{5\sqrt{3}}{6}B +2\log \left( e^A+e^{B/\sqrt{3}} \right), \frac{\sqrt{3}}{2} A+ \frac{1}{2} B \right),\\
  ST  (A ,B) &= \left(  -\frac{1}{2} A +\frac{\sqrt{3}}{2} B,- \frac{\sqrt{3}}{2} A -\frac{1}{2} B \right).
  \end{align*}
  $ST$ acts as $-2\pi /3$ rotation around the origin.
  
  We define
  \[
  \Phi_{1,1}'\coloneqq t\circ \Phi_{1,1} \colon \mathcal{T}(\Sf_{1,1})\rightarrow \mathbb{R}^2
  \]
  as an alternate global coordinate of $\mathcal{T}(\Sf_{1,1})$.

  \subsubsection{$\Sf_{0,4}$}
  \label{actofS04}
   \begin{prop}
  $
  \Mod(\Sf_{0,4}) \approx \Psl (2, \mathbb{Z}) \ltimes ( \mathbb{Z}/{2\mathbb{Z}} \times \mathbb{Z}/{2\mathbb{Z}} ).
   $
  \end{prop}
  \begin{proof}
  See \cite{farbmarglit}.
  \end{proof}
  $ \mathbb{Z}/{2\mathbb{Z}} \times \mathbb{Z}/{2\mathbb{Z}} \triangleleft \Psl (2, \mathbb{Z}) \ltimes ( \mathbb{Z}/{2\mathbb{Z}} \times \mathbb{Z}/{2\mathbb{Z}} )$ corresponds to the subgroup of $\Mod(\Sf_{0,4})$ generated by two hyperelliptic involutions. This is the kernel of the action (see \cite{farbmarglit}). Therefore we should consider the action of $\Mod(\Sf_{0,4})/{( \mathbb{Z}/{2\mathbb{Z}} \times \mathbb{Z}/{2\mathbb{Z}} )} \approx \Psl(2,\mathbb{Z})$.
  
  For generators $u,v,w$ of $\pi _1 (\Sf_{0,4})$ as in \ref{parS04} (see Figure \ref{generators}), we can take half-twists around $wv$ and $vu$ as generators of $\Mod(\Sf_{0,4})/{ ( \mathbb{Z}/{2\mathbb{Z}} \times \mathbb{Z}/{2\mathbb{Z}} )}$. We let $\sigma$ and $\tau$ denote half-twists around $vu$ and $wv$. We can suppose that
   \begin{align*}
  \sigma _{\ast}^{-1} (u)&=u,\  \sigma_{\ast}^{-1} (v)=v,\  \sigma_{\ast}^{-1} (w)=(wvu)^{-1},\\
    \tau_{\ast}^{-1}(u)&=(wvu)^{-1},\  \tau_{\ast}^{-1}(v)=v,\  \tau_{\ast}^{-1}(w)=w
  \end{align*}
   where $\sigma_\ast$ and $\tau_\ast$ are automorphisms of $\pi_1(\Sf_{0,4})$ induced by $\sigma$, $\tau$ (see Figure \ref{dehntwist}). 
  Therefore for any $(\alpha,\beta) \in \mathbb{R}_+^2$,
  \begin{align*}
  \sigma  \tilde{p}_{0,4}(\alpha,\beta) &= (\tilde{\varphi}(\alpha,\beta) ,\tilde{\psi}(\alpha,\beta) ,(\tilde{\omega} \tilde{\psi} \tilde{\varphi})^{-1}(\alpha,\beta)),\\
   \tau  \tilde{p}_{0,4}(\alpha,\beta) &= ((\tilde{\omega} \tilde{\psi} \tilde{\varphi})^{-1}(\alpha,\beta) , \tilde{\psi}(\alpha,\beta) , \tilde{\omega}(\alpha,\beta))
  \end{align*}
  holds. Therefore, 
  \begin{align*}
  \sigma \chi_{\tilde{p}_{0,4}(\alpha,\beta)}&=\chi_{ \sigma  \tilde{p}_{0,4}(\alpha,\beta)}=(\tr \tilde{\varphi} \tilde{\psi}(\alpha,\beta) , \tr \tilde{\psi} (\tilde{\omega} \tilde{\psi} \tilde{\varphi})^{-1}(\alpha,\beta) ,\tr \tilde{\omega} \tilde{\psi}(\alpha,\beta)), \\
  \tau \chi_{  \tilde{p}_{0,4}(\alpha,\beta)}&=\chi_{ \tau  \tilde{p}_{0,4}(\alpha,\beta)}=(\tr \psi (\tilde{\omega} \tilde{\psi} \tilde{\varphi})^{-1}(\alpha,\beta) , \tr \tilde{\psi} \tilde{\omega(\alpha,\beta)}, \tr \tilde{\varphi} \tilde{\psi}(\alpha,\beta)).
  \end{align*}
  Suppose that $(\alpha ', \beta ')=\sigma  (\alpha ,\beta)$ and $(\alpha '', \beta '')=\tau  (\alpha ,\beta)$. Then
  \begin{align*}
  \sigma \chi_{ \tilde{p}_{0,4}(\alpha ,\beta)}&=\chi_{ \tilde{p}_{0,4}(\alpha ', \beta ')} ,\\
    \tau \chi_{ \tilde{p}_{0,4}(\alpha ,\beta)}&=\chi_{ \tilde{p}_{0,4}(\alpha '', \beta '')}
  \end{align*}
  holds. We should express $(\alpha',\beta')$ and $(\alpha'',\beta'')$ in terms of $\alpha$ and $\beta$.
  
  We get
  \begin{align*}
   (\alpha ', \beta ')&=\left(\frac{\alpha\beta}{{{\alpha}^{2}}\beta+{{\alpha}^{2}}+2\alpha+1},\frac{\beta}{{{\left( \alpha\beta+\alpha+1\right) }^{2}}} \right),\\
   (\alpha '',\beta '')&=\left( \frac{\alpha\beta+\alpha+1}{\beta},\alpha\beta+\alpha \right).
  \end{align*}
  Letting $S\coloneqq \sigma \tau^{-2}$ and $T\coloneqq \tau $, we get 
  \[
   S  (\alpha, \beta) = \left(\frac{1}{\alpha},\frac{1}{\beta } \right),\ 
     T (\alpha, \beta) = \left( \frac{\alpha\beta+\alpha+1}{\beta},\alpha\beta+\alpha \right).
  \]
  Then $S^{2} = (ST)^{3} = \mathrm{id}$.
  
  \begin{figure}[h]
  \centering
  \includegraphics[width=7cm]{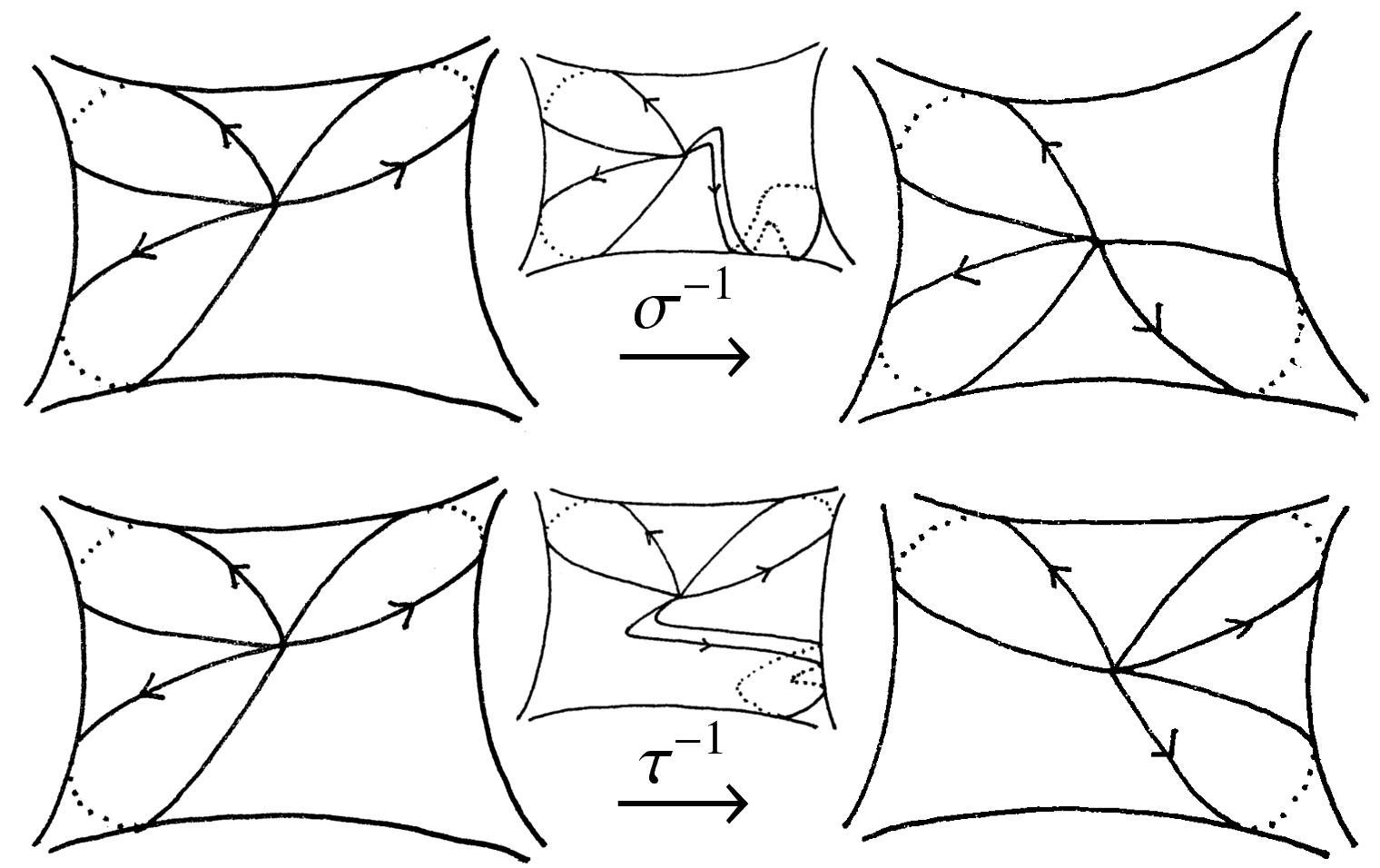}
  \caption{$\Mod(\Sf_{0,4})$.}
  \label{dehntwist}
\end{figure}
  
  Apply a coordinate transformation $t\colon (\alpha , \beta) \mapsto (\log (\alpha), \log (-\beta)) \eqqcolon (A,B)$. Then $t(\mathbb{R}_+^2) =\mathbb{R}^2$ and $\Mod(\Sf_{0,4})$ acts on $\mathbb{R}^2$ as follows:
  \[
  S (A, B) =(-A ,-B),\ 
  T (A, B)=\left( \log \left( e^A +e^{A -B} +e^{-B} \right), \log \left( e^{A +B}+e^A \right) \right).
  \]
  $S$ acts as $\pi$ rotation around the origin.
  
  We define 
  \[
  \Phi_{0,4}' \coloneqq t\circ \Phi_{0,4} \colon \mathcal{T}(\Sf_{0,4}) \rightarrow \mathbb{R}^2
  \]
  as an alternate global coordinate of $\mathcal{T}(\Sf_{0,4})$.

\section{Number of Systoles}
\label{numberofsystole}
 \subsection{Simple Closed Geodesics}
 \subsubsection{Length of Simple Closed Geodesics}
 A simple closed curves on a surface is called \textit{essential} if it is not homotopic to a point or a puncture. There is a bijective correspondence between free homotopy classes of essential curves on a surface $\Sf$ and simple closed geodesics of each hyperbolic surface homeomorphic to $\Sf$. Let $[(\Sigma,h)]=[\rho] \in \mathcal{T}(\Sf)$ and $\gamma$ be an element of $\pi_1(\Sf)$ represented by an essential curve $\tilde{\gamma}$. Then there is the unique simple closed geodesic of $\Sigma$ homotopic to $h(\tilde{\gamma})$ and its length is given by $2 \cosh ^{-1}(|\tr \tilde{\rho} (\gamma)/2|)$. Since the function $x \mapsto 2 \cosh ^{-1}(x)$ is monotone increasing for $x\geq1$, we should consider $|\tr \tilde{\rho}(\gamma)|$. 
For $\gamma \in \pi_1(\Sf)$, we define the map $\Tr_\gamma \colon \mathcal{T}(\Sf) \rightarrow \mathbb{R}_+ $ by $ [\rho] \mapsto |\tr \tilde{\rho} (\gamma)|$ and the set $\mathcal{C}$ by
 \[
 \mathcal{C}= \{\Tr_\gamma \mid \gamma \text{ is represented by an essential curve} \}.
 \]
  There is a natural bijective correspondence between $\mathcal{C}$ and free homotopy classes of essential curves on $\Sf$. 
  Remark that $\Tr_\gamma(m)>2$ for any $\Tr_\gamma \in \mathcal{C}$ and for any $m \in \mathcal{T}(\Sf)$.
 
  For $g\in \Mod(\Sf)$ and $\Tr_\gamma  \in \mathcal{C}$, we let 
  \[
  g \Tr_\gamma \coloneqq \Tr_{g^{-1}_\ast ( \gamma)} 
  \]
   where $g_\ast$ is the automorphism of $\pi_1(\Sf)$ induced by $g$. Then for any $m \in \mathcal{T}(\Sf)$ and for any $\Tr_\gamma \in \mathcal{C}$,
   \[
   g  \Tr_\gamma (m)=\Tr_\gamma (g  m)
   \]
 holds. 
 
 We let $C_m\coloneq \{ \Tr_\gamma(m) \mid \Tr_\gamma \in \mathcal{C} \}$. Suppose that $\gamma \in \pi_1(\Sf)$ is represented by an essential curve $\tilde{\gamma}$. For $[(\Sigma,h)] \in \mathcal{T}(\Sf)$, if $\Tr_\gamma([(\Sigma,h)] )=\min C_{[(\Sigma,h)]} $, then the simple closed geodesic homotopic to $h(\tilde{\gamma})$ is a systole of $\Sigma$.
 \begin{defn*}
 For $\Tr_\gamma \in \mathcal{C}$ and $U\subset \mathcal{T}(\Sf)$, $\Tr_\gamma$ is \textit{minimal} in $U$ if $\Tr_\gamma(m)= \min C_m$ for all $m\in U$.
 \end{defn*}

   Let $u,v$ and $u,v,w$ be generators of $\pi_1(\Sf_{1,1})$ and $\pi_1(\Sf_{0,4})$ as in \ref{parS11} and \ref{parS04}.
   We define $x_0, y_0, z_0 \in \mathcal{C}$ as follows:
   
   In the case of $\Sf_{1,1}$,
   \[
   x_0 \coloneqq \Tr_u, \ y_0 \coloneqq \Tr_v, \  z_0 \coloneqq \Tr_{uv}.
   \]
   
   In the case of $\Sf_{0,4}$,
   \[
   x_0 \coloneqq \Tr_{vw}, \ y_0 \coloneqq \Tr_{wu}, \ z_0 \coloneqq \Tr_{uv}.
   \]
   
   Taking $\tilde{\varphi}, \tilde{\psi}, \tilde{\omega}$ as in \ref{parS11} or \ref{parS04}, we can also write as follows:
   
   In the case of $\Sf_{1,1}$,
     
  
   \begin{align*}
   x_0(m)&=|\tr \tilde{\varphi} (\alpha,\beta)|\\
   &=\left|\frac{\alpha^2+\beta^2+1}{\alpha}\right| \\
   &=\tr \tilde{\varphi} (\alpha,\beta),\\
   y_0(m)&=|\tr \tilde{\psi}(\alpha,\beta)| \\
   &=\left|  \frac{\alpha^2+\beta^2+1}{\beta}  \right| \\
   &=\tr \tilde{\psi}(\alpha,\beta),\\
   z_0(m)&=| \tr \tilde{\varphi}\tilde{\psi}(\alpha,\beta)|\\
   &=\left|\frac{\alpha^2+\beta^2+1}{\alpha \beta}\right|\\
   &=\tr \tilde{\varphi}\tilde{\psi}(\alpha,\beta)
   \end{align*}
   where $m \in \mathcal{T}(\Sf_{1,1})$ and $\Phi_{1,1}(m)=(\alpha,\beta) (\in \mathbb{R}_+^2)$.

   
   In the case of $\Sf_{0,4}$,
 \begin{align*}
   x_0(m)&=|\tr \tilde{\psi} \tilde{\omega}(\alpha,\beta)|\\
   &=\left| 2-\frac{(\alpha+1)^2 ( \beta+1) }{\alpha \beta} \right|\\
   &=-\tr \tilde{\psi} \tilde{\omega}(\alpha,\beta),\\
   y_0(m)&=|\tr \tilde{\omega} \tilde{\varphi}(\alpha, \beta)| \\
   &=\left|  2-\frac{( \alpha+1) ^2 ( \beta+1) }{\alpha}  \right| \\
   &=-\tr \tilde{\omega} \tilde{\varphi}(\alpha, \beta),\\
    z_0(m)&=|\tr \tilde{\varphi} \tilde{\psi} (\alpha, \beta)|\\
    &=\left| 2-\frac{( \alpha+1)^2 (\beta+1)^2}{\beta}\right|\\
    &=-\tr \tilde{\varphi}\tilde{\psi}(\alpha,\beta)
   \end{align*}
   where $m \in \mathcal{T}(\Sf_{0,4})$ and $\Phi_{0,4}(m)=(\alpha,\beta) (\in \mathbb{R}_+^2)$.
   
   \subsubsection{3-Regular Tree and Simple Closed Geodesics}
From now on, we let $\Sf$ denote $\Sf_{1,1}$ or $\Sf_{0,4}$. 
  Consider a 3-regular tree properly embedded in the plane. A \textit{complementary region} is the closure of a connected component of the complement of the tree. Let $\Omega$ denote the set of complementary regions. There is a bijective correspondence between $\Omega$ and free homotopy classes of essential curves on $\Sf$ (see \cite{bowditch1, bowditch2}), and hence between $\Omega$ and $\mathcal{C}$.
  \begin{defn*}
  For $X,Y,Z,W \in \Omega$, $(X,Y ; Z,W)$ indicates that $X,Y,Z,W$ are distinct and $X \cap Y \cap Z \not= \emptyset$ and $X \cap Y \cap W \not= \emptyset$ (see Figure \ref{tritree}).
  \end{defn*}
  
  \begin{figure}[h]
  \centering
  \includegraphics[width=3cm]{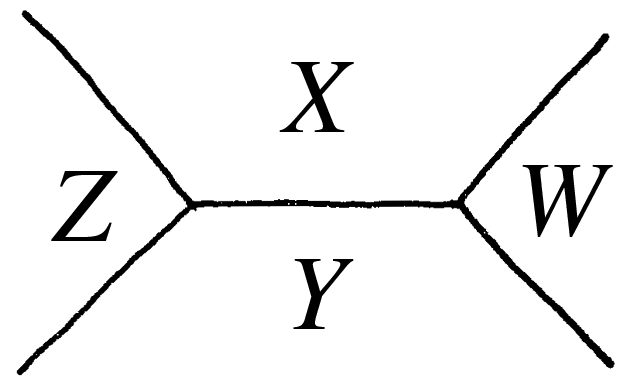}
  \caption{$(X,Y;Z,W)$.}
  \label{tritree}
\end{figure}

   We define the map $\phi \colon \Omega \rightarrow \mathcal{C}$ inductively as follows: 
  \begin{enumerate}
   \item 
   Take $X_0,Y_0,Z_0 \in \Omega$ distinctly so that $X_0\cap Y_0 \cap Z_0 \not=\emptyset$ and fix them.
   \item
   Let $\phi(X_0)=x_0, \phi(Y_0)=y_0, \phi(Z_0)=z_0$.
   \item
   For $X,Y,Z,W \in \Omega$ such that $(X,Y ; Z,W)$, let $\phi(W)=\phi(X) \phi(Y)- \phi(Z)$ in the case of $\Sf_{1,1}$, $\phi(W)=\phi(X) \phi(Y)- \phi(Z)-8$ in the case of $\Sf_{0,4}$.
  \end{enumerate}
  
  We will see the well-definedness and the bijectivity of $\phi$ in Corollary \ref{bij}.
  
  We take $W_0 \in \Omega$ so that $(X_0,Y_0;Z_0,W_0)$. We sometimes write $\phi(\ast)$ by the small letter of $\ast$ (namely $\phi(X) = x,\  \phi(Y) = y,\  \phi(Z) = z,\  \phi(W) = w, \cdots$). 
  
  \begin{lem}
  \label{markoff1}
  For all $X,Y,Z\in \Omega$ such that $X \cap Y \cap Z \not= \emptyset$, in the case of $\Sf_{1,1}$, 
  \[
  x^2+y^2+z^2-xyz=0,
  \]
  and in the case of $\Sf_{0,4}$, 
  \[
  x^2+y^2+z^2-xyz+8(x+y+z)+28=0.
  \]
  \end{lem}
  \begin{proof}
  We prove by induction. $X_0,Y_0,Z_0$ satisfy the lemma. Suppose that $X,Y,Z\in \Omega$ such that $X \cap Y \cap Z \not= \emptyset$ satisfy the lemma. Take $W\in \Omega$ so that $(X,Y;Z,W)$. We should show that $X,Y,W$ satisfy the lemma. In the case of $\Sf_{1,1}$,
  \begin{align*}
  x^2+y^2+w^2-xyw&=x^2+y^2+(xy-z)^2-xy(xy-z)\\
  &=x^2+y^2+z^2-xyz\\
  &=0.
  \end{align*}
  In the case of $\Sf_{0,4}$,
  \begin{align*}
  &x^2+y^2+w^2-xyw+8(x+y+w)+28\\
  =\quad& x^2+y^2+(xy-z-8)^2-xy(xy-z-8)+8(x+y+(xy-z-8))+28\\
  =\quad& x^2+y^2+z^2-xyz+8(x+y+z)+28\\
  =\quad& 0.
  \end{align*}
  \end{proof}
  
    \begin{lem}
    
  For all $m\in \mathcal{T}(\Sf)$ and for all $X,Y,Z\in \Omega$ such that $X \cap Y \cap Z \not= \emptyset$, 
  \[
  x(m),y(m),z(m)>0.
  \]
  \end{lem}
  \begin{proof}
  We prove by induction. $X_0,Y_0,Z_0$ satisfy the lemma. Take $X,Y,Z\in \Omega$ so that $X \cap Y \cap Z \not= \emptyset$. We should show that if $x(m),y(m)>0$, then $z(m)>0$. 
  
  Suppose that $x(m),y(m)>0$ and $z(m) \leq 0$.
  
   In the case of $\Sf_{1,1}$, 
   \[
   x(m)^2+y(m)^2+z(m)^2-x(m)y(m)z(m)>0.
   \]
   This is a contradiction. 
  
  In the case of $\Sf_{0,4}$, 
  \begin{align*}
  &x(m)^2+y(m)^2+z(m)^2-x(m)y(m)z(m)+8(x(m)+y(m)+z(m))+28\\
  &=(x(m)+4)^2+(y(m)+4)^2+(z(m)+4)^2-x(m)y(m)z(m)-20\\
  &>12+(z(m)+4)^2-x(m)y(m)z(m)\\
  &>0.
  \end{align*}
   This is a contradiction.
  \end{proof}
    
  \begin{lem}
  \label{actOnC}
  The following holds:
\[
(S x_0,S y_0,S z_0)=(y_0,x_0,w_0),\  
(ST x_0, ST y_0, ST z_0)=(z_0,x_0,y_0).
\]
\end{lem}
\begin{proof}
In the case of $\Sf_{1,1}$,
\begin{align*}
S_\ast^{-1} (u) &= v, \ S_\ast^{-1} (v) =v^{-1}u^{-1}v , \ S_\ast^{-1} (uv) = u^{-1}v. \\
(ST_\ast)^{-1} (u) &= uv, \ (ST_\ast)^{-1} (v) = v^{-1}u^{-1}v, \ (ST_\ast)^{-1} (uv) = v.
\end{align*} 
where $S_\ast$ and $ST_\ast$ are automorphisms of $\pi_1(\Sf_{1,1})$ induced by $S$ and $ST$.

Therefore, 
\begin{align*}
S x_0&=y_0, \\
 S y_0&=\mathrm{Tr}_{v^{-1}u^{-1}v}\\
 &=x_0,\\
 S z_0&= \mathrm{Tr}_{u^{-1}v}\\
 &=|\tr \tilde{\varphi} (\xi_{1,1}(\cdot)) \tr \tilde{\psi} (\xi_{1,1}(\cdot)) - \tr\tilde{\varphi} \tilde{\psi}(\xi_{1,1}(\cdot))|\\
 &= |x_0y_0-z_0|\\
 &=x_0y_0-z_0\\
 &=w_0.\\
 ST  x_0&=z_0,\\
 ST  y_0&=\mathrm{Tr}_{v^{-1}u^{-1}v}\\
 &=x_0,\\
 ST  z_0 &=y_0.
\end{align*}
We used the trace identity: $\tr AB +\tr AB^{-1}=\tr A \tr B$ in the calculation of $S z_0$. By similar calculation to this, the lemma for the case of $\Sf_{0,4}$ follows.
\end{proof}

   \begin{cor}
   \label{bij}
  $\phi$ is a bijection.
  \end{cor}
  \begin{proof}
    Since there is a bijective correspondence between $\Omega$ and $\mathcal{C}$, we should show that for any $c \in \mathcal{C}$, there exists $F \in \Omega$ such that $\phi(F)=c$ to prove well-definedness and bijectivity. From Lemma \ref{actOnC}, for all $g\in \Mod(\Sf)$, there exists $F\in \Omega$ such that $g x_0=\phi (F)$. Since $\bigcup_{g\in \Mod(\Sf)}g  x_0=\mathcal{C}$, the lemma follows.
  \end{proof}

  \subsubsection{Systoles}
  \label{syst}

   Let $v_0 \in \mathcal{T}(\Sf)$ denote the point such that $x_0 (v_0)= y_0 (v_0)=z_0 (v_0)$ and $v_1 \in \mathcal{T}(\Sf)$ such that $x_0 (v_1)=y_0(v_1)=w_0(v_1)$. Then $ST  v_0=v_0$ and $v_1=T  v_0$. 
   
   Let
   \[
   \mathfrak{s} \coloneqq \{m \in \mathcal{T}(\Sf) \mid z_0(m),w_0(m) \geq x_0(m)=y_0(m)\}.
   \]
   Then $\mathfrak{s}$ is a curve joining $v_0$ and $v_1$. Namely, $\Phi_{1,1}'(\mathfrak{s})=\{(\alpha, 0) \mid \alpha \in [0,2\log 2]\}$ and $\Phi_{0,4}'(\mathfrak{s})=\{(\alpha, 0) \mid \alpha \in [-\log 2,\log 2]\}$.
   
   Then the union:
 \[
  \Gamma  \coloneqq \bigcup _{g \in \Mod(\Sf)} g  \mathfrak{s}
 \]
 is a properly embedded 3-regular tree on $\mathcal{T}(\Sf)$. We draw $\Phi_{1,1}'(\Gamma)$ and $\Phi_{0,4}'(\Gamma)$ in Figure \ref{tritreeS11} and Figure \ref{tritreeS04}.

 \begin{thm}
 \label{thethm}
 For $m \in \mathcal{T}(\Sf)$;
\begin{enumerate}
\item if $m$ is a vertex of $\Gamma$, then $m$ has precisely three different systoles,
\item if $m$ is on an edge of $\Gamma$ but not a vertex, then $m$ has precisely two different systoles,
\item otherwise, $m$ has only one systole.
\end{enumerate}
 \end{thm}
 
 \begin{figure}
  \centering
  \includegraphics[width=8.5cm]{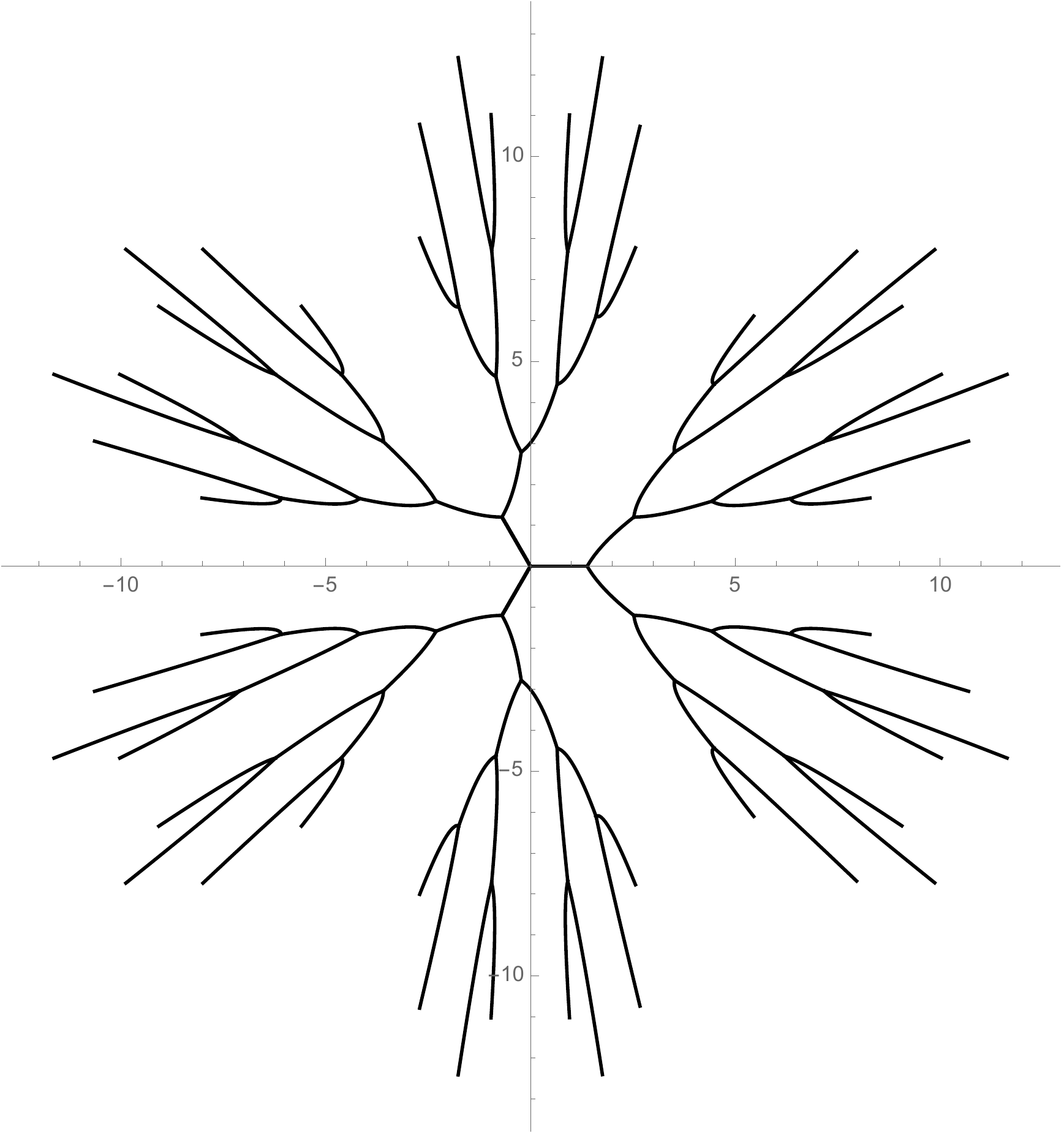}
  \caption{$\Phi_{1,1}'(\Gamma)$.}
  \label{tritreeS11}
  
  \includegraphics[width=8.5cm]{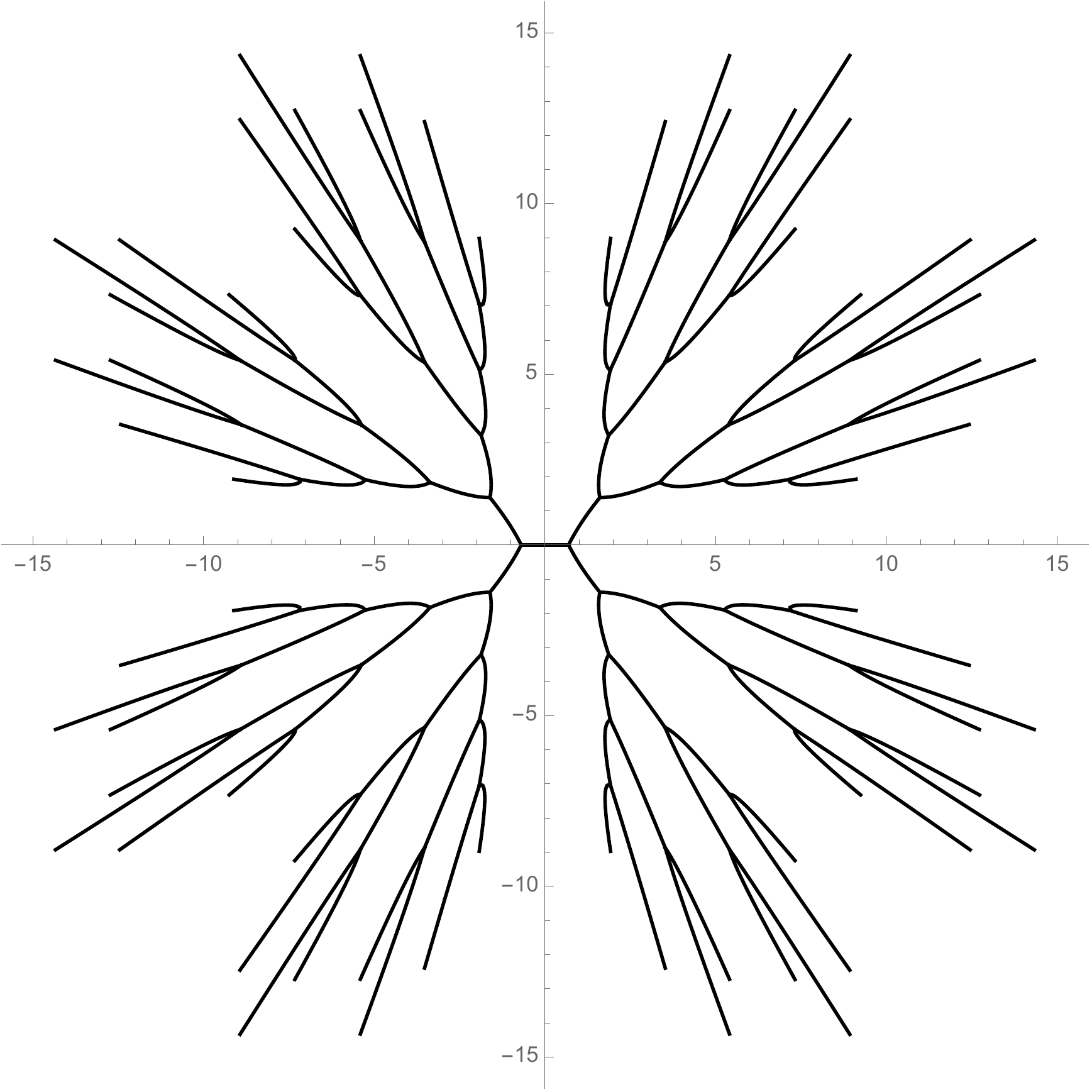}
  \caption{$\Phi_{0,4}'(\Gamma)$.}
  \label{tritreeS04}
\end{figure}


\subsubsection{2-systoles}
Let $v'\in \mathcal{T}(\Sf)$ be the point such that $x_0(v')=y_0(v'),\ z_0(v')=w_0(v')$. Then $S v'=v'$. 

Let 
\[
\mathfrak{t} \coloneqq \{ m\in \mathcal{T}(\Sf) \mid z_0(m),w_0(m)\leq x_0(m)=y_0(m) \} \cup \{v'\}
\]
 and
\[
\Delta_0 \coloneqq \bigcup_{g\in \Mod(\Sf)} g \mathfrak{t},\quad \Delta_1 \coloneqq \bigcup_{g\in \Mod(\Sf)} g \{v_0 \}.
\]
Then $\Delta_0$ and $\Delta_1$ are closed and $\Delta_1 \subset \Delta_0$.
We draw $\Phi_{1,1}'(\Delta_0)$ and $\Phi_{0,4}'(\Delta_0)$ in Figure \ref{deltaS11} and Figure \ref{deltaS04}.
\begin{thm}
\label{thethm2}
 For $m \in \mathcal{T}(\Sf)$;
\begin{enumerate}
\item if $m\in \Delta_1$, then $m$ has precisely three different 2-systoles,
\item if $m\in \Delta_0 \setminus \Delta_1$, then $m$ has precisely two different 2-systoles,
\item otherwise, $m$ has only one 2-systole.
\end{enumerate}
\end{thm}

 \begin{figure}
  \centering
  \includegraphics[width=8.5cm]{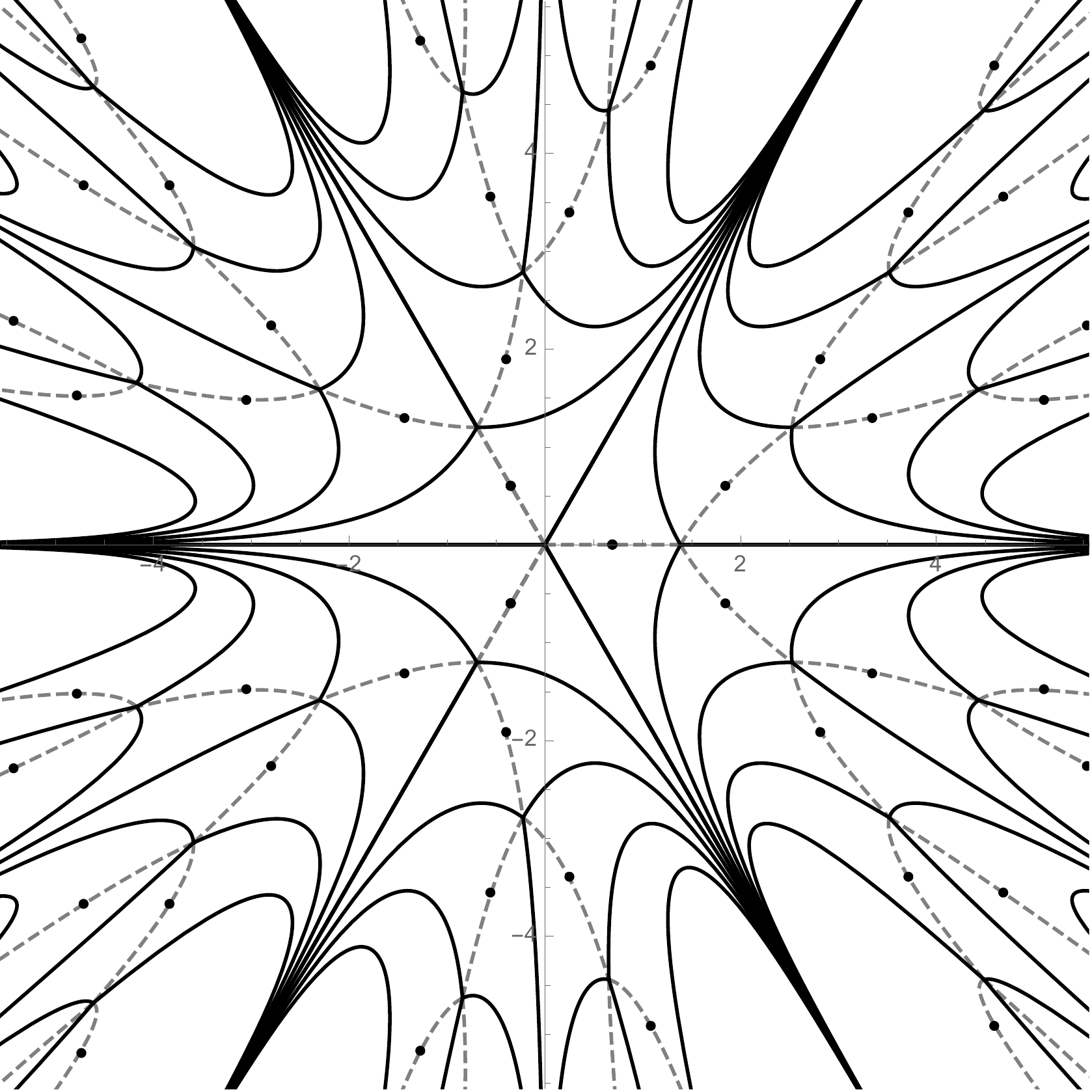}
  \caption{$\Phi_{1,1}'(\Delta_0)$ (the set of intersection points of three half lines is $\Phi_{1,1}'(\Delta_1)$, dashed line is $\Phi_{1,1}'(\Gamma$)).}
  \label{deltaS11}

  \includegraphics[width=8.5cm]{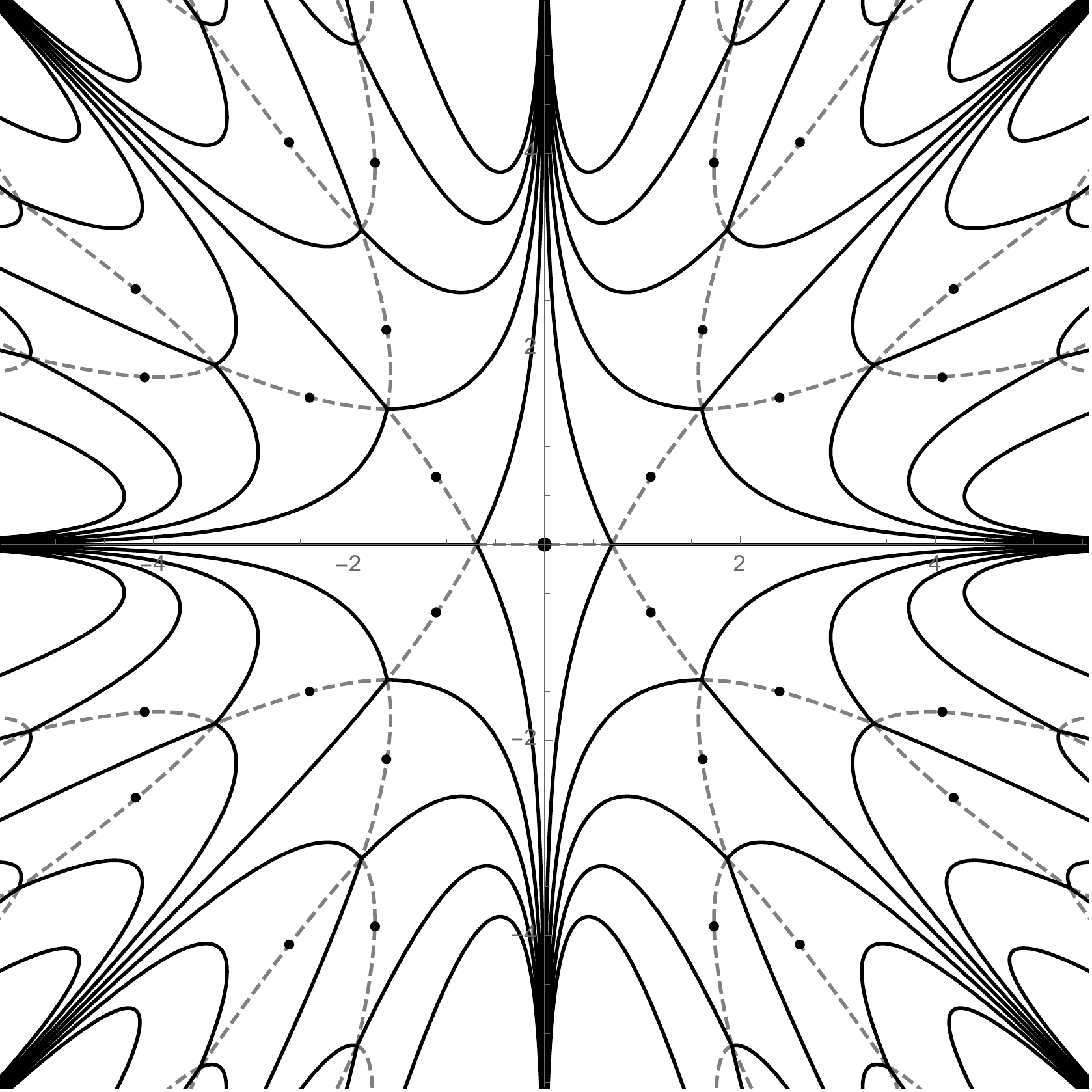}
  \caption{$\Phi_{0,4}'(\Delta_0)$ (the set of intersection points of three half lines is $\Phi_{0,4}'(\Delta_1)$, dashed line is $\Phi_{0,4}'(\Gamma)$).}
  \label{deltaS04}
\end{figure}

 \subsection{Proof of Theorems}
 \label{pfOfThm}
 
 \subsubsection{Proof of Theorem \ref{thethm}}

  \begin{lem}
  \label{markoff2}
  For any triple of real numbers $(x,y,z)$ such that $x,y,z>2$ and $x^2+y^2+z^2-xyz+8(x+y+z)+28=0$, 
  \[
  xy \geq 8+2\sqrt{(x+4)^2+(y+4)^2-4}.
  \]
  
  \end{lem}
  \begin{proof}
 Rewriting the equation as 
 \[
 z^2+z(8-xy)+(x+4)^2+(y+4)^2-4=0,
 \] we get the condition for $z$ to exists:
 \[
  (8-xy)^2-4(x+4)^2-4(y+4)^2+16 \geq 0.
  \]
  Therefore, 
  \[
  xy-8\geq 2 \sqrt{(x+4)^2+(y+4)^2-4}.
  \] 
  \end{proof}

\begin{lem}
For $X,Y,Z,W_1,W_2 \in \Omega$ such that $(X,Z;W_1,Y), (Y,Z;W_2,X)$ (see Figure \ref{tri2}), 
\[
z(m)>x(m),y(m) \Rightarrow w_1(m),w_2(m)>z(m).
\]

\end{lem}
\begin{proof}
Suppose that $z(m)>x(m),y(m)$. 

In the case of $\Sf_{1,1}$, 
\begin{align*}
w_1(m)-z(m)&=x(m)z(m)-y(m)-z(m)\\
&>z(m)(x(m)-2)\\
&>0.
\end{align*} 
Similarly, $w_2(m)-z(m)>0$.

In the case of $\Sf_{0,4}$, by Lemma \ref{markoff1} and Lemma \ref{markoff2},
  \begin{align*}
  w_1(m)-z(m)&=x(m)z(m)-y(m)-8-z(m)\\
  &\geq 2 \sqrt{(x(m)+4)^2+(z(m)+4)^2-4}-y(m)-z(m)\\
  &>2 \sqrt{(z(m)+4)^2+32}-2z(m)\\
  &>0.
  \end{align*}
  Similarly, $w_2(m)-z(m)>0$.
\end{proof}

\begin{figure}[h]
  \centering
  \includegraphics[width=4.5cm]{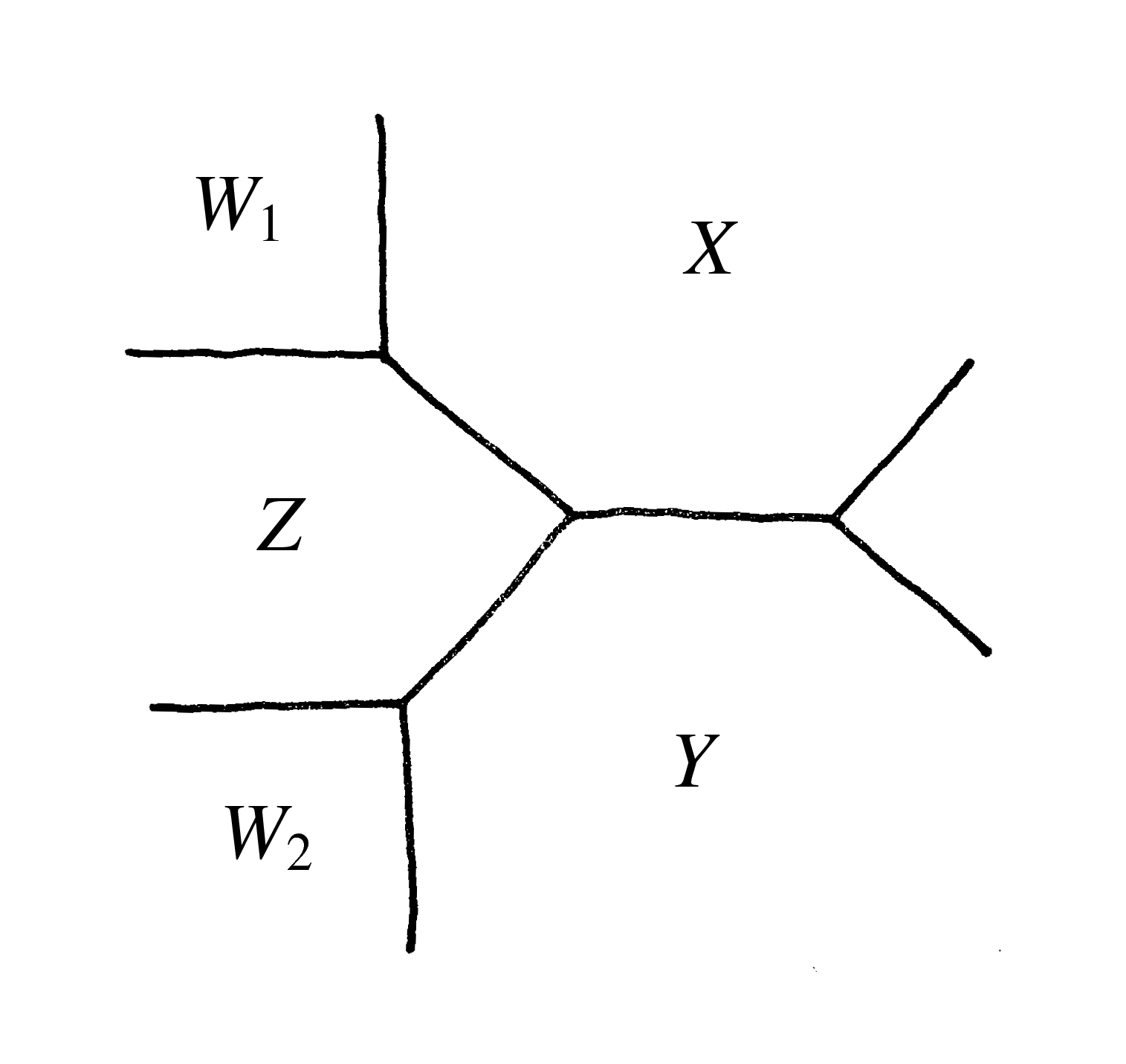}
  \caption{$(X,Z;W_1,Y),(Y,Z;W_2,X)$.}
   \label{tri2}
   \end{figure}

Suppose that $X,Y,Z,W \in \Omega$ and $(X,Y ; Z,W)$. Then $\Omega$ is divided into $\{X,Y\}$ and two subset of $\Omega$, the left hand side of $\{X,Y\}$ and the right hand side of it. We let $\Omega '_{X,Y;Z}$ denote one of these that contains $Z$.  
\begin{cor}
\label{ineq1}
For any $F \in \Omega'_{X,Y;Z}$, 
\[
z(m)>x(m),\ y(m) \Rightarrow \phi (F)(m)>x(m),y(m).
\]

\end{cor}
\begin{proof}
By induction.
\end{proof}
\begin{lem}
\label{ineq2}
Suppose that $X,Y,Z \in \Omega$ and $X\cap Y \cap Z \not= \emptyset$. For any $F\in \Omega \setminus \{X,Y,Z\}$, 
\[
x(m)=y(m)=z(m) \Rightarrow \phi(F)(m)>x(m).
\]

\end{lem}
\begin{proof}
Take $X, Y, Z,W_1, W_2, W_3 \in \Omega$ so that $(X,Y;Z,W_1)$, $(X,Z;Y,W_2)$ and $(Y,Z;X,W_3)$ (see Figure \ref{tri6}).
Suppose that $x(m)=y(m)=z(m)$.

In the case of $\Sf_{1,1}$, 
\begin{align*}
w_1(m)-x(m)&=w_2(m)-x(m)\\
&=w_3(m)-x(m)\\
&=x(m)y(m)-z(m)-x(m)\\
&=x(m)(x(m)-2)\\
&>0.
\end{align*}
 By Corollary \ref{ineq1}, $\phi(F)(m)>x(m)$ for any $F \in \Omega \setminus \{X,Y,Z\}$.

In the case of $\Sf_{0,4}$, by Lemma \ref{markoff1}, 
\[
x(m)=y(m)=z(m)=7.
\] 
Therefore, 
\begin{align*}
w_1(m)-x(m)&=w_2(m)-x(m)\\
&=w_3(m)-x(m)\\
&=34\\
&>0.
\end{align*}
 By Corollary \ref{ineq1}, $\phi(F)(m)>x(m)$ for any $F \in \Omega \setminus \{X,Y,Z\}$.
\end{proof}

\begin{figure}[h]
   \centering
   \includegraphics[width=4.5cm]{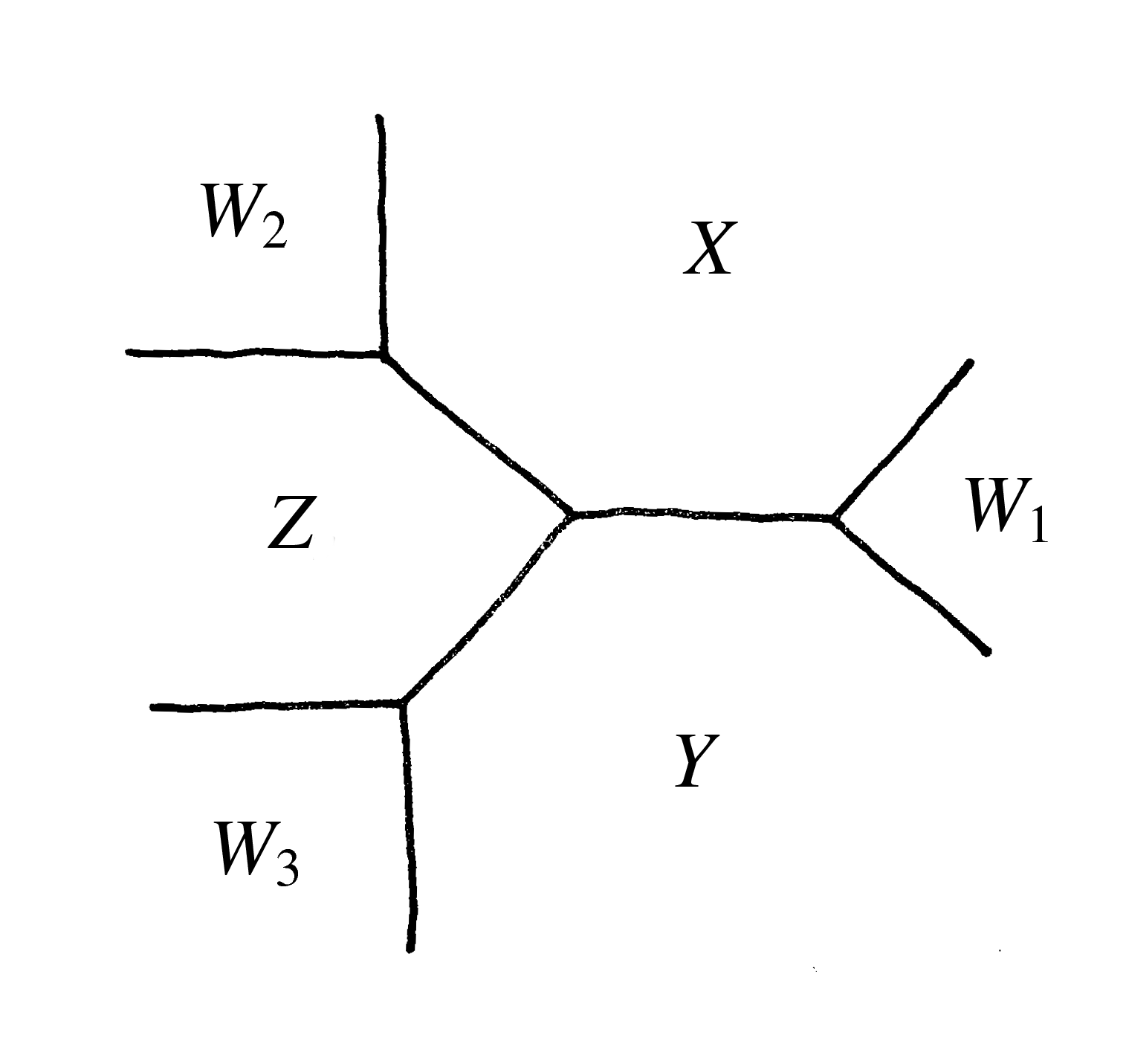}
  \caption{$(X,Y;Z,W_1),(X,Z;Y,W_2),(Y,Z;X,W_3)$.}
  \label{tri6}
  \end{figure}

\begin{cor}
\label{cor1}
Each vertex of $\Gamma$ has a neighborhood in which the claim of Theorem \ref{thethm} holds. 
\end{cor}
\begin{proof}
 Since elements of $\mathcal{C}$ are continuous functions, there exists $U$, a neighborhood of $v_0$  such that $\min C_m = \min \{x_0(m), y_0(m), z_0(m)\}$ for any $m \in U$. It is easy to verify that 
\[
x_0(m)=y_0(m)<z_0(m)\ \text{or}\  y_0(m)=z_0(m)<x_0(m)\  \text{or}\  z_0(m)=x_0(m)<y_0(m)
  \]
   for any $m \in U \cap \Gamma \setminus \{ v_0\}$ and 
   \[
   x_0(m)<y_0(m),z_0(m)\  \text{or}\  y_0(m)<z_0(m),x_0(m)\  \text{or}\  z_0(m)<x_0(m),y_0(m)
   \]
    for any $m\in U \setminus \Gamma$ by simple calculation. Therefore the claim of Theorem \ref{thethm} holds in $U$.
\end{proof}

\begin{lem}
\label{adj}
For $m \in \mathcal{T}(\Sf)$, suppose that $Z_{m}$ is an element of $\Omega$ such that $\phi(Z_m)(m)=\min C_m$. Then for any $F\in \Omega$,
\[
\phi(Z_{m})(m)=\phi(F)(m) \Rightarrow Z_m \cap F \not= \emptyset.
\]
\end{lem}
\begin{proof}
We fix $m \in \mathcal{T}(\Sf)$ and write $z_m\coloneqq \phi(Z_m)$. We should show that $F\in \Omega$ such that $F\cap Z_m =\emptyset$ and $\phi(F)(m)=z_m(m)$ is only $Z_m$.

 Take $X,Y,W_1,W_2,W_3 \in \Omega$ so that $(X,Y;Z_m,W_1)$, $(X,Z_m;Y,W_2)$ and $(Y,Z_m;X,W_3)$ (see Figure \ref{tri3}). Consider four cases:
\begin{enumerate}
\item 
$x(m)=y(m)=z_m(m)$.
\item
$x(m)>y(m)=z_m(m)$.
\item
$y(m)>x(m)=z_m(m)$.
\item
$x(m),y(m)>z_m(m)$.
\end{enumerate}

In (1), by Lemma \ref{ineq2}, $F\in \Omega$ such that $z_m(m)=\phi(F)(m)$ is just $X$ or $Y$ or $Z_m$.

In (2), by the hypothesis, $\phi(F)(m)>y(m)=z_m(m)$ for any $F \in \Omega'_{Y,Z_m;X}$.
If $w_3(m)=z_m(m)$, then $F\in \Omega$ such that $z_m(m)=\phi(F)(m)$ is only $Y$ or $W_3$ or $Z_m$ as in (1). If $w_3(m)>z_m(m)$, then $\phi(F)(m)>y(m)=z_m(m)$ for any $F \in \Omega'_{Y,Z_m;W_3}$ and $F\in \Omega$ such that $z_m(m)=\phi(F)(m)$ is only $Y$ or $Z_m$.

In (3), do a similar argument to (2).

In (4), in the case of $\Sf_{1,1}$,
\begin{align*}
w_1(m)-x(m)&=x(m)y(m)-z_m(m)-x(m)\\
&>x(m)(y(m)-2)\\
&>0.
\end{align*}
Similarly $w_1(m)-y(m)>0$. 

In the case of $\Sf_{0,4}$, 
\begin{align*}
w_1(m)-x(m)&=x(m)y(m)-z_m(m)-8-x(m)\\
&>x(m)y(m)-2x(m)-8\\
&\geq 2\sqrt{(x(m)+4)^2+(y(m)+4)^2-4}-2x(m)\\
&>0.
\end{align*}
 Similarly $w_1(m)-y(m)>0$. Therefore $\phi(F)(m)>z_m(m)$ for any $F \in \Omega'_{X,Y;W_1}$.

If there exists $F\in \Omega \setminus \{Z_m\}$ such that $F\cap Z_m \not= \emptyset$ and $\phi(F)(m)=z_m(m)$, do a similar argument to (2). Otherwise, for all $X',Y',W' \in \Omega$ such that $(X',Y';Z_m,W')$, $x'(m),y'(m)>z_m(m)$ holds. Thus $w'(m)>x'(m),y'(m)>z_m(m)$ (do the same argument as when we showed $w_1(m)>x(m),y(m)$). Therefore $F\in \Omega$ such that $z_m(m)=\phi(F)(m)$ is only $Z_m$.
\end{proof}

\begin{figure}[h]
  \centering
   \includegraphics[width=4.5cm]{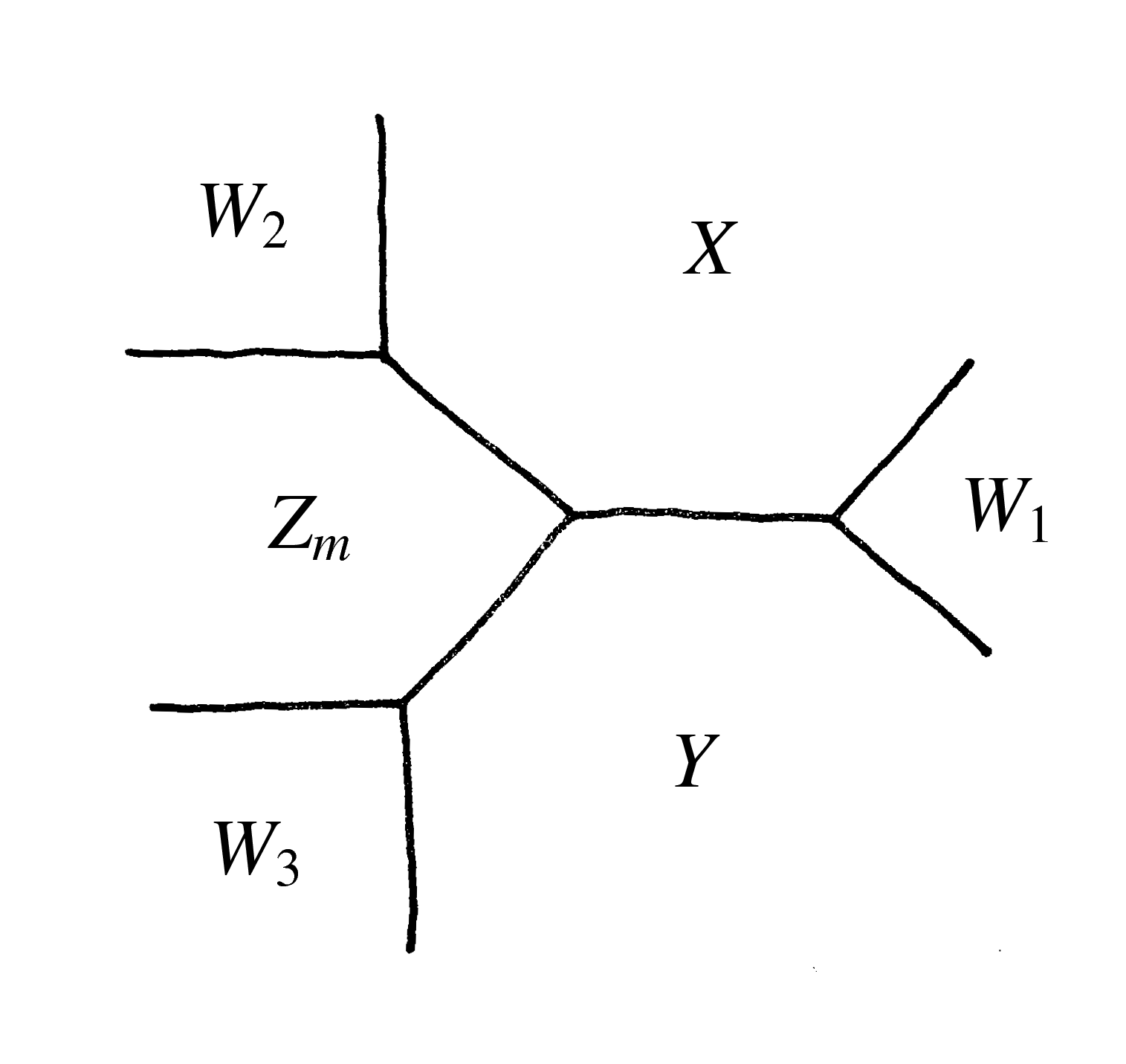}
  \caption{$(X,Y;Z_m,W_1),(X,Z_m;Y,W_2),(Y,Z_m;X,W_3)$.}
  \label{tri3}
  \end{figure}

\begin{lem}
\label{ins}
For any $F\in \Omega \setminus \{ X_0,Y_0\}$ and for any $m\in \mathfrak{s} \setminus \{v_0,v_1\} $,
\[
\phi (F)(m)> x_0(m)=y_0(m).
\]
\end{lem}
\begin{proof}
First, we prove that there do not exist $F\in \Omega  \setminus \{ X_0,Y_0\}$ and $m\in \mathfrak{s} \setminus \{v_0,v_1\}$ such that $\phi (F)(m)= x_0(m)=y_0(m)$. 

By Corollary \ref{cor1}, there exists $U$, a neighborhood of $v_0$ such that only $x_0$ and $y_0$ are minimal in $U\cap \mathfrak{s} \setminus \{ v_0 \}$. Then only $x_0$ and $y_0$ are minimal in $(U \cup T U)\cap \mathfrak{s} \setminus \{v_0,v_1\}$.

Suppose that there exist $F\in \Omega  \setminus \{ X_0,Y_0\}$ and $m\in \mathfrak{s} \setminus \{v_0,v_1\}$ such that $\phi (F)(m)= x_0(m)=y_0(m)$. Then we can take such $F$ and $m$ additionally so that $x_0(m) = \min C_m$. 
Thus by Lemma \ref{adj}, $F\cap X_0 \not= \emptyset$ and $F\cap Y_0 \not= \emptyset$. Hence $F$ is $Z_0$ or $W_0$. However, $m \in \mathcal{T}(\Sf)$ such that $z_0(m)=x_0(m)=y_0(m)$ or $w_0(m)=x_0(m)=y_0(m)$ is only $v_0$ or $v_1$. This is a contradiction. 

Secondly, suppose that there exist $F\in \Omega  \setminus \{ X_0,Y_0\}$ and $m\in \mathfrak{s} \setminus \{v_0,v_1\}$ such that $\phi (F)(m)< x_0(m)=y_0(m)$. Then by the intermediate value theorem, there exists $m'\in \mathfrak{s} \setminus \{v_0,v_1\}$ such that $\phi (F)(m')= x_0(m')=y_0(m')$. This is a contradiction.
\end{proof}
\begin{cor}
\label{cor2}
There exists a neighborhood of $\Gamma$ in which the claim of Theorem \ref{thethm} holds. 
\end{cor}
\begin{proof}
Let $U$ be a neighborhood of $v_0$ in which the claim of Theorem \ref{thethm} holds (then $T U$ is a neighborhood of $v_1$ in which the claim of Theorem \ref{thethm} holds). We should show that there exists a neighborhood of $\mathfrak{s} \setminus (U \cup T U)$ in which the claim of Theorem \ref{thethm} holds. 

By Lemma \ref{ins}, only $x_0$ and $y_0$ are minimal in $\mathfrak{s} \setminus \{ v_0,v_1\}$. Since elements of $\mathcal{C}$ are continuous functions, there exists $U'$, a neighborhood of $\mathfrak{s} \setminus (U \cup T U)$ such that either $x_0$ or $y_0$ but not both is minimal in $U' \setminus \mathfrak{s}$.
\end{proof}

Let $U$ be a neighborhood of $\mathfrak{s}$ in which the claim of Theorem \ref{thethm} holds. Let $D \subset P$ denote the complementary region of $\Gamma$ such that $x_0$ is minimal in $U\cap D$ (then $\Phi_{1,1}'(D)$ and $\Phi_{0,4}'(D)$ are upper sides of $\Phi_{1,1}'(\mathfrak{s})$ and $\Phi_{1,1}'(\mathfrak{s})$). Then $x_0$ is minimal in $D \cap \bigcup_{g\in \Mod(\Sf)} g U$.

\begin{lem}
\label{noteq}
For any $m \in \mathring{D}$ (the interior of $D$), there does not exist $F\in \Omega \setminus \{X_0\}$ such that $x_0(m)=\phi(F)(m)$. 
\end{lem}
\begin{proof}
Suppose that there exist $F\in \Omega \setminus \{X_0\}$ and $m\in \mathring{D}$ such that $x_0(m)=\phi(F)(m)$. Since $x_0$ is continuous, we can take such $F$ and $m$ additionally so that $x_0(m) = \min C_m$. 
Therefore $X_0\cap F \not= \emptyset$. We take such $F$ and $m$. Since the action of $T$ corresponds to a twist around an essential curve corresponding to $x_0$, $T^k  x_0=x_0$ and $x_0\not= T^k  y_0 \not=T^l  y_0$ for all $k,l \in \mathbb{Z} \ (k \not= l)$. We let $F_k \coloneqq \phi^{-1}(T^k  y_0)$. Since $y_0(v_0)=y_0(v_1)$ and $T v_0=v_1$ holds, $T^k  y_0(T^{-k}  v_0)=T^{k+1}  y_0(T^{-k}  v_0)$. Therefore, $X_0 \cap F_k \cap F_{k+1} \not= \emptyset$. 

Thus, there exists $k\in \mathbb{Z}$ such that $F=F_k$. Hence $L \coloneqq \{ m \in \mathcal{T}(\Sf) \mid \phi(F)(m)=x_0(m) \}$ contains an edge contained in $\partial D$. On the other hand, by the hypothesis, $L\cap \mathring{D} \not= \emptyset$. Since the action of $T$ fixes $D$, $T^{-k} (L \cap \mathring{D})= (T^{-k} L) \cap \mathring{D}$. However, when we consider $\Phi_{1,1}'(L)$ and $\Phi_{0,4}'(L)$, they are the line $\{(\alpha,\beta) \in \mathbb{R}^2 \mid \beta=0\}$ (remark that $T^{-k} L=\{m \in \mathcal{T}(\Sf) \mid x_0(m)=y_0(m)\}$). They do not have the intersection with $\Phi_{1,1}'(\mathring{D})$ and $\Phi_{0,4}'(\mathring{D})$. This is a contradiction.
\end{proof}
\begin{cor}
\label{cor3}
Surfaces in $\bigcup_{g\in \Mod(\Sf)}g \cdot \mathring{D}$ have only one systole.
\end{cor}
\begin{proof}
We should prove that $x_0(m)<\phi(F)(m)$ for any $m\in \mathring{D}$ and for any $F\in \Omega \setminus \{ X_0\}$.
Suppose that there exist $m\in \mathring{D}$ and $F\in \Omega \setminus \{ X_0\}$ such that $x_0(m)>\phi(F)(m)$. By the intermediate value theorem, there exists $m' \in \mathring{D}$ such that $x_0(m')=\phi(F)(m')$. This is a contradiction.
\end{proof}

From Corollary \ref{cor2} and Corollary \ref{cor3}, Theorem \ref{thethm} follows.


\subsubsection{Proof of Theorem \ref{thethm2}}

We should prove that the claim of the theorem holds in $D$. From the argument in the proof of Theorem \ref{thethm},  $x_0$ is minimal in $D$ and only $x_0$ is minimal in $\mathring{D}$.

\begin{lem}
\label{2syslem1}
Suppose that $Y,Z,W \in \Omega$, $(X_0,Y;Z,W)$ and $m \in \mathring{D}$. 
If $y(m)<w(m),z(m)$, then for any $F\in \Omega \setminus \{X_0,Y\}$,
\[
y(m)<\phi(F)(m).
\] 
If $y(m)=z(m)$, then for any $F\in \Omega \setminus \{X_0,Y,Z\}$, the same inequality holds.
\end{lem}
\begin{proof}
Take $Y,Z,W \in \Omega$ so that $(X_0,Y;Z,W)$.
Suppose that $y(m)<w(m),z(m)$. From Corollary \ref{ineq1}, $y(m)<\phi(F)(m)$ for any $F \in \Omega'_{X_0,Y;Z},\  \Omega'_{X_0,Y;W}$ (remark that $x_0$ is minimal in $D$). Therefore $y(m)<\phi(F)(m)$ for any $F \in\Omega'_{X_0,Y;Z}\cup \Omega'_{X_0,Y;W}=\Omega \setminus \{X_0,Y\}$.

 Take $Y,Z,W_1,W_2 \in \Omega$ so that $(X_0,Z;Y,W_1),(Y,Z;X_0,W_2)$. Suppose that $y(m)=z(m)$. 
 
 In the case of $\Sf_{1,1}$, 
 \begin{align*}
 w(m)&=w_1(m)\\
 &=x_0(m)y(m)-y(m),\\
  w_2(m)&=(y(m))^2 -x_0(m).\\
   \end{align*}
   Therefore
   \begin{align*}
   w(m)-y(m)&=w_1(m)-z(m)\\
   &=y(m)(x_0(m)-2)\\
   &>0, \\
   w_2(m)-y(m)&=(y(m))^2-y(m)-x_0(m)\\
   &>y(m)(y(m)-2)\\
   &>0.
 \end{align*}

In the case of $\Sf_{0,4}$, 
\begin{align*}
w(m)&=w_1(m)\\
&=x_0(m)y(m)-y(m)-8,\\
w_2(m)&=(y(m))^2 -x_0(m)-8.
\end{align*} 
From Lemma \ref{markoff2},
\begin{align*}
w(m)-y(m)&=w_1(m)-z(m)\\
&=x_0(m)y(m)-2y(m)-8\\
&>2\sqrt{(x_0(m)+4)^2+(y(m)+4)^2-4}-2y(m)\\
&>0,\\
w_2(m)-y(m)&=(y(m))^2-y(m)-x_0(m)-8\\
&>x_0(m)y(m)-2y(m)-8\\
&>0.
\end{align*}

Therefore $y(m)<w(m),w_1(m),w_2(m)$. From Corollary \ref{ineq1}, $y(m)=z(m)<\phi(F)(m)$ for any $F\in \Omega'_{X_0,Y;W}\cup \Omega'_{X_0,Z;W_1}\cup \Omega'_{Y,Z,W_2}=\Omega \setminus \{X_0,Y,Z\}$.
\end{proof}

\begin{lem}
\label{2syslem2}
Suppose that $m\in \mathfrak{s}\setminus \{ v_0,v_1 \}$.  For any $F \in \Omega'_{X_0,Y_0;Z_0} \setminus \{Z_0\}$,
\[
z_0(m)<\phi(F)(m),
\]
 and for any $F \in \Omega'_{X_0,Y_0;W_0} \setminus \{W_0\}$,
 \[
  w_0(m)<\phi(F)(m).
 \]
\end{lem}
\begin{proof}
Take $W_1,W_2,W_3,W_4 \in \Omega$ so that $(X_0,Z_0;Y_0,W_1)$, $(Y_0,Z_0;X_0,W_2)$, $(X_0,W_0;Y_0,W_3)$ and $(Y_0,W_0;X_0,W_4)$ (see Figure \ref{tri5}). Suppose that $m\in \mathfrak{s}\setminus \{ v_0,v_1 \}$.

In the case of $\Sf_{1,1}$, 
\begin{align*}
w_1(m)-z_0(m)&=w_2(m)-z_0(m)\\
&=x_0(m)z_0(m)-x_0(m)-z_0(m)\\
&>z_0(m)(x_0(m)-2)\\
&>0.
\end{align*}
 Similarly $w_3(m)-w_0(m)=w_4(m)-w_0(m)>0$ (remark that $x_0(m)=y_0(m)=\min C_m$ if $m \in \mathfrak{s}\setminus \{ v_0,v_1 \}$).

In the case of $\Sf_{0,4}$, from Lemma \ref{markoff2},
\begin{align*}
w_1(m)-z_0(m)&=w_2(m)-z_0(m)\\
&=x_0(m)z_0(m)-x_0(m)-z_0(m)-8\\
&>x_0(m)z_0(m)-2z_0(m)-8\\
&>2\sqrt{(x_0(m)+4)^2+(z_0(m)+4)^2-4}-2z_0(m)\\
&>0.
\end{align*}
Similarly $w_3(m)-w_0(m)=w_4(m)-w_0(m)>0$.

Therefore $z_0(m)>w_1(m),w_2(m)$ and $w_0(m)>w_3(m),w_4(m)$. From Corollary \ref{ineq1}, $z_0(m)<\phi(F)(m)$ for any $F \in \Omega'_{X_0,Z_0;W_1} \cup \Omega'_{Y_0,Z_0;W_2}=\Omega'_{X_0,Y_0;Z_0} \setminus \{Z_0\}$ and $w_0(m)<\phi(F)(m)$ for any $F \in \Omega'_{X_0,W_0;W_3} \cup \Omega'_{Y_0,W_0;W_4}=\Omega'_{X_0,Y_0;W_0} \setminus \{W_0\}$.
\end{proof}

\begin{figure}[h]

 \centering
   \includegraphics[width=5cm]{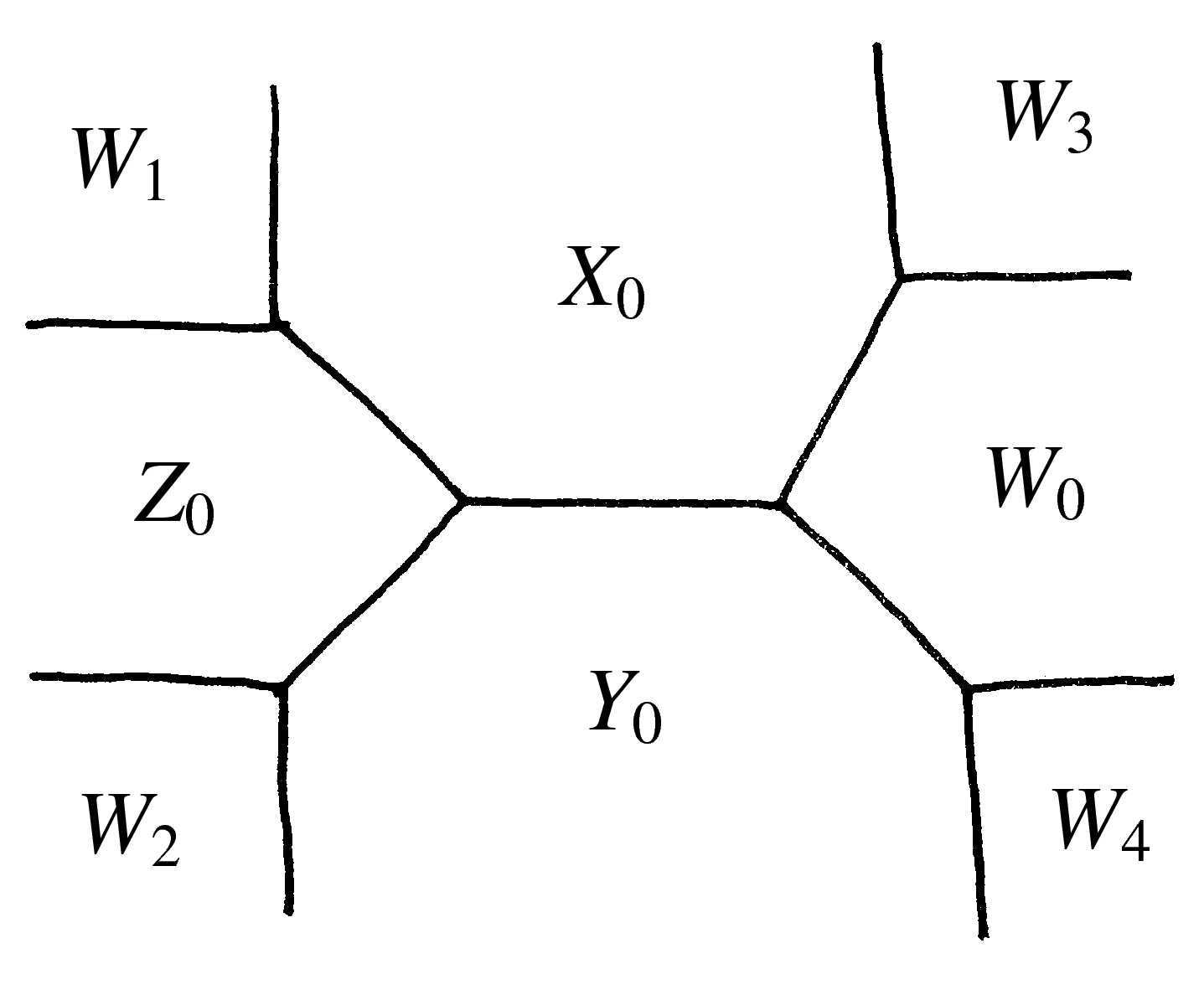}
  \caption{$(X_0,Z_0;Y_0,W_1),(Y_0,Z_0;X_0,W_2),(X_0,W_0;Y_0,W_3),(Y_0,W_0;X_0,W_4)$.}
  \label{tri5}
\end{figure}

\begin{cor}
\label{2syscor1}
If $m \in \Delta_0 \setminus \Delta_1$, then $m$ has precisely two 2-systoles.
\end{cor}
\begin{proof}
We should prove that $m \in \{m'\in \mathcal{T}(\Sf) \mid y_0(m')=z_0(m')\}\cap \mathring{D}$ and $v'$ have only two 2-systoles. From Lemma \ref{2syslem1}, if $m \in \{m'\in \mathcal{T}(\Sf) \mid y_0(m')=z_0(m')\}\cap \mathring{D}$, only two simple closed geodesics corresponding to $Y_0$ and $Z_0$ are 2-systoles of $m$. From Lemma \ref{2syslem2} and the definition of $v'$ (we defined $v'$ by $w_0(v')=z_0(v')$), only two simple closed geodesics corresponding to $W_0$ and $Z_0$ are 2-systoles of $v'$.
\end{proof}

\begin{cor}
\label{2syscor2}
If $m\in {\Delta_0}^c$ (the complement of $\Delta_0$), then $m$ has only one 2-systole.
\end{cor}
\begin{proof}
Suppose that $m\in \mathring{D} \cap {\Delta_0}^c$. From Lemma \ref{2syslem1}, there exists $Z\in \Omega$ such that $X_0\cap Z \not= \emptyset$ and $z(m)<\phi(F)(m)$ for any $F\in \{F' \in \Omega \setminus \{X_0,Z\}  \mid X_0\cap F' \not= \emptyset \}$. Then the simple closed geodesic corresponds to $Z$ is the unique 2-systole.

Suppose that $m\in \mathfrak{s} \setminus \{v_0,v_1\}$. From Lemma \ref{2syslem2}, simple closed geodesics that correspond to $Z_0$ and $W_0$ are 2-systoles. $m \in \mathfrak{s} \setminus \{v_0,v_1\}$ satisfying $z_0(m)=w_0(m)$ is only $v'$. Therefore if $m\in \mathfrak{s} \setminus \{v_0,v_1,v'\}$, $m$ has only one 2-systole.

Hence if $m \in \bigcup_{g\in\Mod(\Sf)} g (\mathring{D} \cap {\Delta_0}^c \cup \mathfrak{s} \setminus \{v_0,v_1,v'\})={\Delta_0}^c$, $m$ has only one 2-systole.
\end{proof}

\begin{lem}
\label{2syslem3}
Suppose that $W_1,W_2 \in \Omega$, $(X_0,Z_0;Y_0,W_1)$ and $(Y_0,Z_0;X_0,W_2)$. For any $F\in\Omega \setminus \{ X_0,Y_0,Z_0,W_0,W_1,W_2 \}$,
\[
w_0(v_0)=w_1(v_0)=w_2(v_0)<\phi(F)(v_0).
\] 
\end{lem}
\begin{proof}
In the case of $\Sf_{1,1}$, 
\begin{align*}
x_0(v_0)&=y_0(v_0)=z_0(v_0)=3,\\
w_0(v_0)&=w_1(v_0)=w_2(v_0)=6.
\end{align*}
 In the case of $\Sf_{0,4}$, 
 \begin{align*}
 x_0(v_0)&=y_0(v_0)=z_0(v_0)=7,\\
  w_0(v_0)&=w_1(v_0)=w_2(v_0)=34.
  \end{align*}
   From Corollary \ref{ineq1}, the lemma follows.
\end{proof}

From Corollary \ref{2syscor1}, Corollary \ref{2syscor2}, and Lemma \ref{2syslem3}, Theorem \ref{thethm2} follows.

\subsubsection*{Acknowledgements}
The author would like to thank Professor Sadayoshi Kojima and Mitsuhiko Takasawa for many valuable advices. He also would like to thank Yi Wan. The part of once-punctured torus case is based on his master thesis \cite{wan}.